\documentclass[11pt]{article}
\usepackage{amsmath,amssymb}

\usepackage{times}
\usepackage{bm}
\usepackage{natbib}
\usepackage{algorithm}
\usepackage{xargs}[2008/03/08]
\usepackage{mathrsfs}
\usepackage{graphicx}
\usepackage{rotating,subfigure}
\usepackage[flushleft]{threeparttable}
\usepackage{multirow}
\usepackage{enumitem} %%%enumitem

\usepackage{star}

\usepackage[colorlinks,
linkcolor=blue,
anchorcolor=blue,
citecolor=blue
]{hyperref}

\def\T{{ \mathrm{\scriptscriptstyle T} }}

\def\##1\#{\begin{align}#1\end{align}}
\def\$#1\${\begin{align*}#1\end{align*}}

\def\T{{ \mathrm{\scriptscriptstyle T} }} %%%transpose operator

\newcommand{\Rom}[1]{\text{\uppercase\expandafter{\romannumeral #1\relax}}}

\newcommand\mc[1]{\mathcal{#1}}
\newcommand\mbf[1]{\mathbf{#1}}
\newcommand\mbb[1]{\mathbb{#1}}
\newcommand\rank[1]{\mathrm{rank}(#1)}
\def\P{\mathbb{P}} %%%transpose operator
\def\E{\mathbb{E}} %%%transpose operator
\def\Pa{\mathbb{P}_{(\sigma, \xi, \bm{\beta})}}
\def\Ea{\mathbb{E}_{(\sigma, \xi, \bm{\beta})}}
\def\Pn{\mathbb{P}_{\star}}
\def\En{\mathbb{E}_{\star}}

\usepackage{geometry}
\begin{document}

\title{ \LARGE Bayesian high-dimensional linear regression\\with generic spike-and-slab priors}

\author{Bai Jiang\thanks{ByteDance AI Lab, 250 Bryant Street, Mountain View, CA 94041, USA; Email: \url{bai.jiang@bytedance.com}.} ~ and~Qiang Sun\thanks{Department of Statistical Sciences, University of Toronto, 100 St. George Street, Toronto, ON M5S 3G3, Canada; Email: \url{qsun@utstat.toronto.edu}.}}

%qsun.ustc@gmail.com

\date{}

\maketitle

\vspace{-0.25in}

% typeset the title of the contribution
\begin{abstract}
Spike-and-slab priors are popular Bayesian solutions for high-dimensional linear regression problems. Previous theoretical studies on spike-and-slab methods focus on specific prior formulations and use prior-dependent conditions and analyses, and thus can not be generalized directly. In this paper, we propose a class of generic spike-and-slab priors and develop a unified framework to rigorously assess their theoretical properties. Technically, we provide general conditions under which generic spike-and-slab priors can achieve the nearly-optimal posterior contraction rate and the model selection consistency. Our results include those of \citet{narisetty2014bayesian} and \citet{castillo2015bayesian} as special cases.
\end{abstract}

\section{Introduction}\label{sec:1}

Consider the linear regression model
\begin{equation} \label{model}
\mbf{Y} = \mbf{X}\bm{\beta}^\star +\sigma_\star \bm{\varepsilon},
\end{equation}
where $\mbf{X} \in\RR^{n\times p}$ is a deterministic design matrix, $\bm{\beta}^\star\in \RR^{p}$ is a vector of unknown regression coefficients, $\sigma_\star > 0$ is unknown standard deviation, and $\bm{\varepsilon} \sim \mathcal{N}(\bm{0}, \mbf{I}_n)$ is a standard normal vector. We are interested in parameter estimation and model selection in the high-dimensional regime where $p\gg n$ and a small number $s$ of covariates contribute to the response. Formally, we assume the index set of non-zero regression coefficients $\xi_\star := \{j: \beta^\star_j \ne 0\}$ is $s$-sparse. The goals are to estimate unknown parameters $(\sigma_\star, \bm{\beta}^\star)$ and to identify the true sparse model (index set) $\xi_\star$.

For this high-dimensional linear regression problem, many methods have been proposed from the Bayesian perspective. They commonly encourage sparsity of the regression coefficients by adopting suitable priors \citep{park2008bayesian, hans2009bayesian, carvalho2010horseshoe, griffin2012structuring, armagan2013generalized, narisetty2014bayesian, castillo2015bayesian, bhattacharya2015dirichlet, martin2017empirical, rovckova2018spike}. These priors can be mainly categorized into two categories: shrinkage priors and  spike-and-slab priors.

The shrinkage priors are directly motivated by the equivalence between regularized maximum likelihood estimators in the frequentist framework (among others, \citealp{tibshirani1996regression, fan2001variable, zou2006adaptive, candes2007dantzig,  zhang2010nearly}) and \textit{maximum a posteriori} estimators in the Bayesian framework. Examples include the Bayesian lasso prior \citep{park2008bayesian, hans2009bayesian}, the horseshoe prior \citep{carvalho2010horseshoe, polson2012local}, the correlated normal-gamma prior \citep{griffin2012structuring}, the double Pareto prior \citep{armagan2013generalized}, the Dirichlet-Laplace prior \citep{bhattacharya2015dirichlet} and the spike-and-slab lasso prior \citep{rovckova2018spike}, to name a few. Recently \citet{song2017nearly} provide sufficient conditions for generic shrinkage priors to achieve the nearly-optimal parameter estimation rate and the model selection consistency.

The spike-and-slab priors are hierarchical priors which naturally arise from the probabilistic consideration of the high-dimensional linear regression model \eqref{model}. A generic spike-and-slab prior $\Pi(\sigma, \xi, \bm{\beta})$ takes the form of 
\begin{equation} \label{prior}
\begin{split}
\sigma^2 &\sim g(\sigma^2),\\
\xi & \sim \pi(\xi),~~~\xi \subseteq \{1,\dots,p\},\\
\beta_j|\sigma^2 &\sim h_0(\beta_j/\sigma)/\sigma,~~~\forall~j \not \in \xi, \\
\beta_j|\sigma^2 &\sim h_1(\beta_j/\sigma)/\sigma,~~~\forall~j \in \xi,
\end{split}
\end{equation}
where $g$ is a density function over $(0,\infty)$, $\xi \subseteq \{1,\dots,p\}$ indexes all possible $2^p$ subset models, $\pi(\xi)$ is a model selection prior introducing model sparsity, $h_0$ is a ``spike" distribution for modeling negligible coefficients (e.g., the Dirac measure at $0$) and $h_1$ is a ``slab'' distribution for modeling significant coefficients. 
This generic spike-and-slab prior of form \eqref{prior} dates back to the Dirac-and-slab priors \citep{mitchell1988bayesian, george1993variable, johnson2012bayesian} and mixture Gaussian priors \citep{ishwaran2005spike} in the small-$p$-large-$n$ setup. Later, in the high-dimensional regime, \citet{narisetty2014bayesian} showed that a mixture Gaussian prior with shrinking and diffusing scale hyper-parameters consistently selects the true sparse model. \citet{castillo2015bayesian} studied both the parameter estimation rate and the model selection consistency for Dirac-and-Laplace priors. Other recent works also consider the correlated Gaussian distribution as the slab prior \citep{yang2016computational, martin2017empirical}.

Despite of these works, theoretical properties for generic spike-and-slab priors of form \eqref{prior} remain unclear. Previous works usually narrow their focuses down to specific choices of spike-and-slab formulations, such as the mixture Gaussian prior \citep{narisetty2014bayesian} and the Dirac-and-Laplace prior \citep{castillo2015bayesian}, and conduct theoretical assessments under conditions that cope with their choices of formulations. Consequently, their analyses rely on various conditions and are not generalizable to other spike-and-slab priors.

Regarding the model selection prior $\pi(\xi)$, a popular choice is the independently and identically distributed (i.i.d.) Bernoulli prior \citep{narisetty2014bayesian}, in which each covariate $j$ is independently selected into the model $\xi$ with probability $1/p$. \citet[Assumption 1]{castillo2015bayesian} considered a class of model selection priors exponentially decreasing on model size, which we refer to as the Castillo-Schimdt-Vaart priors or simply the CSV priors. Note that the i.i.d. Bernoulli prior does not belong to the class of CSV priors. Regarding the spike and slab distributions, popular choices include the Laplace or the Gaussian distribution and the Dirac measure at $0$ as the spike distribution only \citep{narisetty2014bayesian, castillo2015bayesian, rovckova2018spike}. Although some combinations of the above-mentioned model selection priors, spike priors and slab priors have been recognized, the potential of other combinations for Bayesian high-dimensional linear regression has been overlooked.

On the other hand, different conditions on the eigen-structure of the Gram matrix $\mbf{X}^\T\mbf{X}$ have been proposed and assumed. Examples include conditions on the minimum non-zero eigenvalue \citep[\texttt{MNEV}]{narisetty2014bayesian}, the minimum restricted eigenvalue \citep[\texttt{MREV}]{castillo2015bayesian}, and the minimum sparse eigenvalue \citep[\texttt{MSEV}]{song2015high}. The \texttt{MREV} condition has been widely assumed for frequentist methods \citep{bickel2009simultaneous, raskutti2010restricted, fan2018lamm}. As for Bayesian methods, however, it is unclear how these different conditions relate to each other and whether the results built on one of them can transfer to the other.

In this paper, we develop a unified framework to analyze Bayesian methods with generic spike-and-slab priors. This framework could not only facilitate theoretical assessments of a broad class of spike-and-slab priors, but also unifies seemingly different conditions for existing spike-and-slab methods.

First, for the parameter estimation task, we give a high-level condition for model selection prior $\pi(\xi)$, which is satisfied by both CSV priors and the i.i.d. Bernoulli prior. Another interesting finding is that \texttt{MNEV}, \texttt{MREV} and \texttt{MSEV} are lower bounds of a quantity, which we call the \textit{minimum united eigenvalue} (\texttt{MUEV}) and denote by $\lambda$. A positive $\lambda$ suffices for the Bayesian spike-and-slab methods to succeed. Second, for the model selection task, we show that, under the commonly-seen beta-min condition \citep[Corollary 7.6]{buhlmann2011statistics}, the generic spike-and-slab prior selects overfitted models that overshoot the true model size by no more than a constant factor. Finally, we identify two more technical conditions, which enable eliminating false discoveries in the overfitted models and consistently selecting true sparse model. Conditions for previous spike-and-slab methods \citep{castillo2015bayesian, narisetty2014bayesian} are shown to be special cases of our conditions tailored to their specific spike-and-slab priors.

The rest of the paper proceeds as follows. Section \ref{sec2} introduces a class of generic spike-and-slab priors. Section \ref{sec2.5} builds up the posterior contraction of parameters $(\sigma, \bm{\beta})$ upon the new local eigenvalue condition relating to \texttt{MUEV} $\lambda$. Section \ref{sec3} presents additional conditions and theorems for the model selection task. Section \ref{sec4} sketches the proofs of theorems, with proofs of technical lemmas deferred to the appendix. Section \ref{sec5} concludes the paper with a discussion.

\subsubsection*{Notation}
For the high-dimensional linear regression model \eqref{model}, both $p$ and $s$ should be understood as sequences of $n$, i.e., $p=p_n$ and $s=s_n$, although their subscripts $n$ are omitted. Similarly, for the spike-and-slab prior $\Pi(\sigma, \xi, \bm{\beta})$ specified by \eqref{prior}, $\pi$, $h_0$ and $h_1$ should be understood as sequences of distributions $\pi_n$, $h_{0n}$ and $h_{1n}$. Let
$$\epsilon_n := \sqrt{s \log p/n},$$
be the nearly-optimal\footnote{The optimal $\ell_2$-estimation error rate is $\sqrt{s \log(p/s)/n}$; see \citet{raskutti2011minimax} and \citet{su2016slope} among others.} $\ell_2$-estimation error rate for estimating $\bm{\beta}^\star$.

For  $\bm{\beta} \in \mathbb{R}^p$, let $\beta_j$ denote its $j$-th component and $\bm{\beta}_\xi$ denote its sub-vector consisting of coordinates in the subset $\xi \subseteq \{1,\dots,p\}$. We also call the index set $\xi$ a model in the context of model selection. For a vector $\bm{v}$, let $\Vert \bm{v} \Vert_q$ with $q \in [1,\infty]$ denote its $\ell_q$-norm. For $\ell_2$-norm, we omit the subscript $2$ and write $\Vert \bm{v} \Vert$ for simplicity.

We write $\mbf{X}_j$ to denote the $j$-th column of $\mbf{X}$ and $\mbf{X}_\xi$ to denote the sub-matrix consisting of columns indexed by  $\xi \in \{1,\dots,p\}$. For a model $\xi$, let $|\xi|$ be its cardinality, and $\rank{\xi}$ be the column rank of $\mbf{X}_\xi$. The model $\xi$ is said to be of full rank if $\rank{\xi} = |\xi|$. Let $\mc{F}$ denote the set of all full-rank models. Formally
$$\mc{F} := \{\xi \subseteq \{1,\dots,p\}: \rank{\xi} = |\xi|\}.$$
Let $\mbf{P}_\xi$ denote the projection matrix onto the column space of $\mbf{X}_\xi$. Note that, in case of $\xi \in \mc{F}$,
$\mbf{P}_\xi = \mbf{X}_\xi\left(\mbf{X}_\xi^\T\mbf{X}_\xi\right)^{-1}\mbf{X}_\xi^\T$. For $\xi \in \mc{F}$, let $\mbf{X}_\xi^\dagger =  \left(\mbf{X}_\xi^\T\mbf{X}_\xi\right)^{-1}\mbf{X}_\xi^\T$ be the left pseudo-inverse of $\mbf{X}_\xi$.

For a symmetric matrix $\mbf{A}$, we write its largest eigenvalue as $\lambda_{\max}(\mbf{A})$ and its smallest eigenvalue as $\lambda_{\min}(\mbf{A})$. For two symmetric matrices $\mbf{A}$ and $\mbf{B}$, $\mbf{A} \ge \mbf{B}$ or $\mbf{B} \le \mbf{A}$ means $\mbf{A}-\mbf{B}$ is positive semi-definite. For two positive sequences $a_n$ and $b_n$, $a_n \prec b_n$ or $b_n \succ a_n$ means $\limsup_{n \to \infty} a_n/b_n = 0$; $a_n \preceq b_n$ or $b_n \succeq a_n$ means $\limsup_{n \to \infty} a_n/b_n < \infty$; $a_n \asymp b_n$ means $a_n \preceq b_n$ and $a_n \succeq b_n$; $a_n \lesssim b_n$ or $b_n \gtrsim a_n$ means $a_n > b_n$ for sufficiently large $n$. We write $o(1)$ to denote an arbitrarily small positive constant. %so, for constant $c > 0$, $c - o(1)$ should be understood as a constant that is arbitrarily close to $c$.

Let $\Pa$ and $\Ea$ denote the  measure and expectation associated with model \eqref{model} with parameters $(\sigma, \xi, \bm{\beta})$. Write $\P_{(\sigma_\star, \xi_\star, \bm{\beta}^\star)}$ and $\E_{(\sigma_\star, \xi_\star, \bm{\beta}^\star)}$ as $\Pn$ and $\En$ for simplicity. When the probability $\Pa(E)$ of an event $E$ does not depend on $(\sigma, \xi, \bm{\beta})$, we write $\Pa(E)$ as $\P(E)$.

\section{A Class of Generic Spike-and-slab Priors} \label{sec2}
Throughout the paper, we tacitly assume that the response vector $\mbf{Y}$ has been centered at $0$, and thus include no intercept term in the linear regression model \eqref{model}. Each covariate $\mbf{X}_j$ is centered and standardized such that $\Vert \mbf{X}_j \Vert  = \sqrt{n}$. The true standard deviation $\sigma_\star \in (0,\infty)$ is fixed. The true model $\xi_\star$ is of full rank. We focus on the asymptotic regime where $p > n$ but $\epsilon_n = \sqrt{s\log p/n} \to 0$.

Our following assumption specifies a class of generic spike-and-slab priors.
\begin{assumption}[On Prior] \label{asm1}\mbox{}
\begin{enumerate}[label=(\alph*)]
\item Variance prior: The density function of variance $g(\sigma^2)$ is continuous and positive for any $\sigma^2 \in (0,+\infty)$.
\item Model selection prior: The model $\xi$ is selected \textit{a priori} with probability being proportional to $\pi(\xi)1\{\xi \in \mc{F}\}$; And, $\pi(\xi)$ satisfies $\pi(\emptyset) \asymp 1$, and with constants $A_1, A_2>0$,
$$\pi(\xi_\star) \ge p^{-A_1s},~~~\sum_{\xi:~|\xi| > t} \pi(\xi) \le p^{-A_2t},~\forall t\ge 1$$
for sufficiently large $n$\footnote{Recall that both $p = p_n$ and $s=s_n$ are sequences of $n$.}.
\item Spike prior: The sequence $z_{0n}$ such that 
$\int 1\left\{|z| > z_{0n} \right\} h_0(z)dz = e^{-n}$
satisfies
$$z_{0n} \prec \frac{1}{p}\sqrt{\frac{\log p}{n}}.$$
\item Slab prior: For $z_{1n} := \max_{j \in \xi_\star} |\beta^\star_j/\sigma_\star|+\epsilon_n$ and some constant $A_3>0$, the slab density function $h_1(z)$ satisfies
$$\inf_{z:~|z| \le z_{1n}} h_1(z) \succeq p^{-A_3}.$$
\end{enumerate}
\end{assumption}

Assumption \ref{asm1}(a) is satisfied when $g$ is the inverse-gamma density function. If $\sigma_\star^2$ is known to take values in an interval, we could set $g$ to be a continuous and positive density function over that interval, e.g., the truncated inverse-gamma density function. 

Assumption \ref{asm1}(b) requires the model selection prior $\pi(\xi)$ to assign a sufficient mass to the true sparse model $\xi_\star$, and to downweight large models. The following examples show that this requirement is met by the commonly-used i.i.d. Bernoulli prior \citep{narisetty2014bayesian}, and the CSV priors including the complextity prior and the Binomial-Beta prior \citep{castillo2012needles}. Appendix \ref{appendix:A} collects the detailed proofs.  %\scomment{we probably want to discuss the following examples in greater details.}

\begin{example}[I.I.D. Bernoulli Prior] \label{example:bernoulli}
The i.i.d. Bernoulli prior $\pi(\xi)$ selects each index $j$ into $\xi$ with probability $1/p$. This prior meets the condition for $\pi$ in Assumption \ref{asm1}(b) with $\pi(\emptyset) \approx e^{-1}$, and any $A_1 > 1$ and $A_2 = 1$. The deduction of $A_2=1$ needs a novel tail probability inequality due to \citet[Theorem 1]{pelekis2016lower}.
\end{example}

\begin{example}[CSV Prior] \label{example:csv}
Suppose $w(t)$ is a discrete distribution over possible model sizes $t \in \{0,\dots,p\}$, and
$$B_1 p^{-B_3} w(t-1) \le w(t) \le B_2 p^{-B_4} w(t-1),~~~t=1,\dots,p,$$
with constants $B_1, B_2, B_3, B_4 > 0$. The prior $\pi(\xi)$ for model selection given by
$$\pi(\xi) = w(|\xi|){p \choose |\xi|}^{-1}$$
meets Assumption \ref{asm1}(b) with any $A_1 > B_3+1$ and any $A_2 < B_4$.
\end{example}

%\begin{example}[Complexity Prior for Model Size]
%The complexity prior on the model size $t$ is of the form
%$$w(t) \propto c^{-t}p^{-at}, ~~~t=0,1,\dots,p.$$
%for constants $a,c>0$. This prior meets the requirement on $w(t)$ in Example \ref{example:csv}.
%\end{example}

%\begin{example}[Binomial-Beta Prior for Model Size]
%The Binomial-Beta prior on the model size $t$ is of the form
%$$w(t|r) = {p \choose t} r^t (1-r)^{p-t},~~~t=0,1,\dots,p,$$
%where the Binomial success rate $r$ follows a Beta distribution $\mathrm{Beta}(1,p^u)$ with some $u > 1$. This prior meets the requirement on $w(t)$ in Example \ref{example:csv}.
%\end{example}

Assumption \ref{asm1}(c) requires the spike distribution to be the Dirac measure at zero or degenerate to it at an appropriate rate. This rate would ensure the aggregated signal of inactive covariates with \textit{a priori} regression coefficients to be negligible
$$\max_{\xi \supseteq \xi_\star} \frac{\Vert \mbf{X}_{\xi^c} \bm{\beta}_{\xi^c} \Vert}{\sqrt{n}} \prec \sqrt{\frac{\log p}{n}} \le \epsilon_n.$$

Assumption \ref{asm1}(d) avoids excessive thinness of the slab distribution around the true regression coefficients, which would otherwise cause the slab prior to miss true non-zero regression coefficients. In case of the Laplace slab distribution $h_1(z) = (\rho_n/2)\exp(-\rho_n|z|)$, the choice of the inverse-scale hyper-parameter $p^{-A_3'} \preceq \rho_n \preceq A_3'' \log p / z_{1n}$ fulfills Assumption \ref{asm1}(d) with $A_3 = A_3' + A_3''$. In case of the Gaussian slab distribution $h_1(z) = (2\pi\tau_{1n}^2)^{-1/2}\exp[-z^2/(2\tau_{1n}^2)]$, the choice of the variance hyper-parameter $A_3' z_{1n}^2 / \log p \preceq \tau_{1n}^2 \preceq p^{A_3''}$ fulfills Assumption \ref{asm1}(d) with $A_3 = 1/A_3' + A_3''/2$.

%\scolor{Assumption \ref{asm1}(d) avoids excessive thinness of the slab distribution around the true regression coefficients, which would otherwise cause the slab prior to miss  true non-zero regression coefficients. This assumption holds if $\max_{j \in \xi_\star} |\beta^\star_j| \preceq \rho_n \log p$ in case of the Laplace slab distribution $h_1(z) = (\rho_n/2)\exp(-\rho_n|z|)$, or $\max_{j \in \xi_\star} |\beta^\star_j| \preceq \tau_{1n} \sqrt{\log p}$ in case of the Gaussian slab distribution $h_1(z) = (2\pi\tau_{1n}^2)^{-1/2}\exp[-z^2/(2\tau_{1n}^2)]$.}

\section{Posterior Contraction} \label{sec2.5}

It is impossible to estimate the coefficients $\bm{\beta}_\star$ in the high-dimensional linear regression model \eqref{model} without
conditions on the Gram matrix $\mbf{X}^\T\mbf{X}$. Indeed, the Gram matrix is non-invertible in the high-dimensional regime, rendering an unidentifiability issue. A common remedy is to assume some kind of ``local invertibility'' of the Gram matrix. We formalize this idea in the following definition and assumption.
\begin{definition}[Minimum United Eigenvalue (\texttt{MUEV})] \label{def1}
The minimum united eigenvalue of order $t$ for the design matrix $\mbf{X}$ is defined as
$$\texttt{MUEV}(t) := \min_{\xi \in \mc{F}:~|\xi \cup \xi_\star| \le t} \lambda_{\min}(\mbf{X}_{\xi \cup \xi_\star}^\T\mbf{X}_{\xi \cup \xi_\star}/n),$$
where $\mc{F}$ is the set of all full-rank models $\xi$.
\end{definition}

\begin{assumption}[\texttt{MUEV} Condition]\label{asm2}
There exists constant $K>0$ such that
$$\lambda = \lambda(K) := \texttt{MUEV}((K+1)s) > 0.$$
\end{assumption}

We also collect other local eigenvalues used in the literature. 
\begin{definition}[Minimum Restricted Eigenvalue (\texttt{MREV})] \label{def2}
The minimum restricted eigenvalue of order $t$ (with parameter $\alpha \ge 1$) is defined as
$$\texttt{MREV}(t) := \min_{\xi:~|\xi| \le t} \inf_{\bm{\beta}}\left\{\frac{\bm{\beta}^\T\mbf{X}^\T\mbf{X}\bm{\beta}}{n\Vert \bm{\beta}\Vert^2}: \bm{\beta} \ne \mbf{0}, \Vert \bm{\beta}_{\xi^c} \Vert_1 \le \alpha \Vert \bm{\beta}_{\xi}\Vert_1\right\}.$$
\end{definition}

\begin{definition}[Minimum Sparse Eigenvalue (\texttt{MSEV})] \label{def3}
The minimum sparse eigenvalue of order $t$ is defined as
$$\texttt{MSEV}(t) := \min_{\xi:~|\xi| \le t} \lambda_{\min} (\mbf{X}_{\xi}^\T\mbf{X}_{\xi}/n).$$
\end{definition}

\begin{definition}[Minimum Non-zero Eigenvalue (\texttt{MNEV})] \label{def4}
The minimum non-zero eigenvalue of order $t$ is defined as
$$\texttt{MNEV}(t) := \min_{\xi:~|\xi| \le t} \lambda_{\min}^\texttt{N} (\mbf{X}_{\xi}^\T\mbf{X}_{\xi}/n),$$
where $\lambda_{\min}^\texttt{N}(\mbf{A})$ denotes the minimum non-zero eigenvalue of a symmetric matrix $\mbf{A}$.
\end{definition}

The next lemma discusses the relation of \texttt{MUEV} to other local eigenvalues of the Gram matrix. It states that the \texttt{MUEV} condition is the weakest among other possible conditions defined upon \texttt{MSEV}, \texttt{MREV} and \texttt{MNEV}. Note that the premise for \texttt{MUEV} $\ge$ \texttt{MNEV} in the lemma was assumed in the original paper of \texttt{MNEV}; see \citep[Condition 4.5]{narisetty2014bayesian}.

\begin{lemma} \label{lem:eigenvalues}
For any $t > s$,
$$\texttt{MUEV}(t) \ge \texttt{MSEV}(t) \ge \texttt{MREV}(t).$$
If $(\mbf{I} - \mbf{P}_\xi)\mbf{X}_{\xi_\star} \ne \mbf{0}$ for any $\xi \not \supseteq \xi_\star$ of size $|\xi| \le t$, then
$$\texttt{MUEV}(t) \ge \texttt{MNEV}(t).$$
\end{lemma}

Now we are ready to present our main results regarding the posterior contraction rate of $\bm{\beta}$ in terms of $\ell_2$-norm.

\begin{theorem}[Posterior Contraction] \label{thm1}
Suppose Assumption \ref{asm1} and Assumption \ref{asm2} hold with $A_1 + A_3 + 1 < A_2 K$. For any constants $M_1,M_2 > \sqrt{8\max\{A_2,1\}K}$, the posterior distribution $\Pi(\sigma, \xi, \bm{\beta}|\mbf{X},\mbf{Y})$ concentrates on the subset of the parameter space
\begin{equation*}
\widehat{\Theta} = \left\{(\sigma, \xi, \bm{\beta}):
\begin{split}
& \frac{\sigma^2}{\sigma_\star^2} \in \left[\frac{1-M_1\epsilon_n}{1+M_1\epsilon_n}, \frac{1+M_1\epsilon_n}{1-M_1\epsilon_n}\right],\\
&|\xi \setminus \xi_\star| \le Ks,~\xi \in \mathcal{F},\\
&\max_{j \not \in \xi} |\beta_j| \le \sigma z_{0n},\\
& \Vert \bm{\beta} - \bm{\beta}^\star\Vert \le M_2\sigma_\star\epsilon_n/\sqrt{\lambda}\\
\end{split}
\right\}
\end{equation*}
in the sense that, with some constants $C_1, C_2$,
$$\Pn \left(\Pi(\widehat{\Theta}|\mbf{X},\mbf{Y}) \ge 1 - e^{-C_1s\log p}\right) \gtrsim 1-e^{-C_2s\log p}.$$
\end{theorem}

This theorem establishes the posterior contraction rate of regression coefficients $\bm{\beta}$ in terms of $\ell_2$-norm. Roughly speaking, the posterior distribution puts almost all mass in an $\epsilon_n$-ball centering at true coefficients $\bm{\beta}^\star$, with high probability. The posterior contraction rate is a standard metric to evaluate estimation accuracy of Bayesian approaches \citep{ghosal2000convergence, shen2001rates}.

Two appealing byproducts of Theorem \ref{thm1} are the adaptivity of the posterior distribution to the unknown variance $\sigma_\star^2$, and the unknown sparsity level $s$. The working variance $\sigma^2$ accurately estimates $\sigma_\star^2$ up to a relative error of order $\epsilon_n$. The working model $\xi$ does not overshoot the true model size $s = |\xi_\star|$ by more than a constant factor $K$.

Additionally, Theorem \ref{thm1} allows $\lambda = \lambda(K)$ in Assumption \ref{asm2} to decrease to zero as $n \to \infty$, providing broader applicabilities. In this case, Theorem \ref{thm1} gives the posterior contraction rate $\epsilon_n/\sqrt{\lambda}$. When $\lambda$ is of constant order, this rate is near optimal. %When $\lambda$ diminishes as $n \to \infty$, this rate is meant to read for a sequence of $\lambda$.

\section{Model Selection} \label{sec3}

The task of consistent model selection is harder than the task of parameter estimation, and thus requires more assumptions. Indeed, if some coordinates of $\bm{\beta}^\star$ are too close to zero, then no method can detect these nearly-zero coordinates as being non-zero. In such case, the posterior distribution may select only a subset of the true model and possibly other coordinates (with nearly-zero coefficients). At the same time, the parameter estimation could be still accurate due to the negligible coefficient values of these false discoveries. To avoid these extreme cases that cause the model selection task to fail, we need some kind of ``beta-min'' condition \citep{buhlmann2011statistics} on the minimal value of the true coefficients $\bm{\beta}^\star$.

\begin{assumption}[Beta-min Condition] \label{asm3}
These exists a constant $M_3 > 0$ such that 
$$\min_{j \in \xi_\star} |\beta^\star_j| \ge \frac{M_3 \sigma_\star \epsilon_n}{\sqrt{\lambda}}.$$
\end{assumption}
%Note that this required rate of $\min_{j \in \xi_\star} |\beta^\star_j| \ge \epsilon_n$ can be derived by the ``identifiability'' condition of \citet[Condition 4.5]{narisetty2014bayesian}.

Under this assumption, the posterior distribution would select all active covariates. This insight is formalized in the following theorem.

\begin{theorem}[Overfitted Model Selection] \label{thm2}
If assumptions in Theorem \ref{thm1} as well as Assumption \ref{asm3} with $M_3 > \sqrt{8A_3K}$ hold, then the posterior distribution $\Pi(\sigma, \xi, \bm{\beta}|\mbf{X},\mbf{Y})$ concentrates on the subset of the parameter space
\begin{equation*}
\widetilde{\Theta} = \left\{(\sigma, \xi, \bm{\beta}):
\begin{split}
& \frac{\sigma^2}{\sigma_\star^2} \in \left[\frac{1-M_1\epsilon_n}{1+M_1\epsilon_n}, \frac{1+M_1\epsilon_n}{1-M_1\epsilon_n}\right],\\
&|\xi \setminus \xi_\star| \le Ks,~\xi \in \mathcal{F},~\xi \supseteq \xi_\star,\\
&\max_{j \not \in \xi} |\beta_j| \le \sigma z_{0n},\\
& \Vert \bm{\beta}_\xi - \bm{\beta}^\star_\xi \Vert \le M_2\sigma_\star\epsilon_n/\sqrt{\lambda}\\
\end{split}
\right\}.
\end{equation*}
in the sense that, with some constants $C_3, C_4$,
$$\Pn \left(\Pi(\widetilde{\Theta}|\mbf{X},\mbf{Y}) \ge 1 - e^{-C_3s\log p}\right) \gtrsim 1-e^{-C_4s\log p}.$$
\end{theorem}

However, Theorem \ref{thm2} does not guarantee the elimination of false discoveries in the overfitted models. To achieve the true sparse model, we need to bound
$$\underbrace{\frac{\Pi(\widetilde{\Theta} \cap\{\xi = \gamma\}|\mbf{X},\mbf{Y})}{\Pi(\widetilde{\Theta} \cap\{\xi = \xi_\star\}|\mbf{X},\mbf{Y})}}_\text{posterior ratio} =  \underbrace{\frac{\Pi(\widetilde{\Theta}|\mbf{X},\mbf{Y},\xi=\gamma)}{\Pi(\widetilde{\Theta}|\mbf{X},\mbf{Y},\xi=\xi_\star)}}_\text{conditional posterior ratio} \times \underbrace{\frac{\pi(\gamma)}{\pi(\xi_\star)}}_\text{prior ratio}$$
for each overfitted model $\gamma \supset \xi_\star$ with $|\gamma| \le (K+1)s$. The two terms on the right-hand side of the last display are the conditional posterior and prior ratios between models $\gamma$ and $\xi_\star$ respectively. Given some continuity condition of the slab distribution $h_1$, the posterior ratio can be bounded as
$$\text{conditional posterior ratio} \le 2\left(\frac{\sqrt{2\pi}\sup_{z} h_1(z) p^{1+o(1)}}{\sqrt{n\lambda}}\right)^{|\gamma|-s}.$$
For many diffusing slab distributions that have diminishing maximum values $\sup_z h_1(z) \to 0$, the conditional posterior ratio is upper bounded by the diffusing rate of $h_1(z)$ after proper normalization. %and the regression setup specified by $(n, p, \lambda)$. 
The prior ratio is solely determined by the model selection prior $\pi(\xi)$. Combining these pieces together yields the following theorem for model selection consistency of spike-and-slab methods.

\begin{theorem}[Consistent Model Selection] \label{thm3}
Suppose assumptions in Theorem \ref{thm2} and the following two conditions hold.
\begin{enumerate}[label=(\alph*)]
\item The log-slab function $\log h_1(z)$ is $L_n$-Lipschitz continuous on $[-z_{1n},+z_{1n}]$ with $s L_n \prec \sqrt{n \log p}$.
\item With some small constant $\eta > 0$,
$$r_n := \max_{j=1}^{Ks}\left[\frac{\pi(|\xi|=s+j)}{\pi(\xi_\star)}\right]^{1/j} \times \frac{\sup_{z} h_1(z)p^{1+\eta}}{\sqrt{n\lambda}} < 1.$$
\end{enumerate}
Then, with constants $C_5, C_6$,
$$\Pn \left(\Pi(\xi = \xi_\star|\mbf{X},\mbf{Y}) \ge 1 - e^{-C_5s\log p} - r_n\right) \gtrsim 1-p^{-C_6}.$$
Consequently, if $r_n \prec 1$ then $\Pi(\xi = \xi_\star|\mbf{X},\mbf{Y})$ converges to $1$ in expectation and in probability.
\end{theorem}

The expression of rate $r_n$ precisely characterizes the roles of the model selection prior $\pi(\xi)$ and the diffusing slab distribution $h_1(z)$ in a successful Bayesian model selection procedure. Conditions of \citet{ narisetty2014bayesian, castillo2015bayesian} are special cases of our general conditions in Theorem \ref{thm3}.

\begin{example}[{\citealp{narisetty2014bayesian}}]\label{example:1}
For the Gaussian slab distribution with the variance parameter $\tau_{1n}^2$ and the i.i.d. Bernoulli prior of model selection (see Example \ref{example:bernoulli}),
$$L_n \asymp \tau_{1n}^{-2},~~~\max_{j=1}^{Ks}\left[\frac{\pi(|\xi|=s+j)}{\pi(\xi_\star)}\right]^{1/j} \le 1.$$
(as they assume fixed true coefficients $\bm{\beta}^\star$, and thus $z_{1n} \asymp 1$). In this setup, conditions (a) and (b) in Theorem \ref{thm3} turn to be
$$n \tau_{1n}^2 \lambda \prec p^{2+2\eta}.$$
This is the rate assumed by \citet[the first display of Section 2.1]{narisetty2014bayesian}.
\end{example}

\begin{example}[{\citealp{castillo2015bayesian}}]\label{example:2}
For the Laplace slab distribution with the inverse-scale parameter $\rho_n$ and the CSV prior of model selection (see Example \ref{example:csv}),
$$L_n = \rho_n,~~~\max_{j=1}^{Ks}\left[\frac{\pi(|\xi|=s+j)}{\pi(\xi_\star)}\right]^{1/j} \le (K+1)s \times B_3 p^{-B_4}.$$
In this setup, conditions (a) and (b) in Theorem \ref{thm3} are satisfied if
$$s \rho_n \sqrt{\log p/n} \prec 1,~~~s < p^\eta,~~~\rho_n \le 4\sqrt{n\log p}$$
for some $\eta < B_4 - 1$. These sufficient conditions have been used by \citet[Corollary 1]{castillo2015bayesian}.
\end{example}

\section{Proofs of Theorems} \label{sec4}
In this section, we sketch proofs of Theorem \ref{thm1}, Theorem \ref{thm2} and Theorem \ref{thm3}. Technical lemmas and their proofs are collected in the appendix. 
\subsection{Proof of Theorem \ref{thm1}}
The claimed inequality of Theorem \ref{thm1} is equivalent to
\begin{equation} \label{eqn:thm1}
\Pn \left(\Pi(\widehat{\Theta}^c|\mbf{X},\mbf{Y}) \ge e^{-C_1s\log p}\right) \lesssim e^{-C_2s\log p}.
\end{equation}
Our central technique to prove \eqref{eqn:thm1} is the following lemma, which is borrowed from \citet[Lemma 6]{barron1998information} and \citet[Lemma A4]{song2017nearly}.

\begin{lemma} \label{lem:barron}
Consider a parametric model $\{P_{\bm{\theta}}\}_{\bm{\theta} \in \Theta}$, and a data generation $\mbf{D}$ from the true parameter $\bm{\theta}_\star \in \Theta$. Let $\Pi(\bm{\theta})$ be a prior distribution over $\Theta$. If
\begin{enumerate}[label=(\alph*)]
\item $\Pi(\Theta_0) \le \delta_0$,
\item there exists a test function $\phi(\mbf{D})$ such that
$$\sup_{\bm{\theta} \in \Theta_\text{test}} \mbb{E}_{\bm{\theta}} [1 - \phi(\mbf{D})] \le \delta_1, \quad \mbb{E}_{\bm{\theta}_\star} [\phi(\mbf{D})] \le \delta_1',$$
\item and
$$\mbb{P}_{\bm{\theta}_\star}\left(\int_{\Theta} \frac{P_{\bm{\theta}}(\mbf{D})}{P_{\bm{\theta}_\star}(\mbf{D})}d\Pi(\bm{\theta}) \le \delta_2\right) \le \delta_2',$$
\end{enumerate}
then for any $\delta_3$,
$$\mbb{P}_{\bm{\theta}_\star}\left(\Pi(\Theta_0 \cup \Theta_\text{test} |\mbf{D}) \ge \frac{\delta_0+\delta_1}{\delta_2\delta_3}\right) \le \delta_1' + \delta_2' + \delta_3.$$
\end{lemma}
Specifically, in the linear regression model \eqref{model},
$$\mbf{D} = (\mbf{X}, \mbf{Y}), ~~~\bm{\theta} = (\sigma, \xi, \bm{\beta}),~~~P_{\bm{\theta}}(\mbf{D}) = \mc{N}(\mbf{Y}|\mbf{X}\bm{\beta}, \sigma^2\mbf{I}).$$
The high-level idea is that, to get the posterior concentration on a desired subset $\widehat{\Theta}$ of the parameter space, one can split the set of undesired parameter values $\widehat{\Theta}^c$ as two subsets $\Theta_0$ and $\Theta_\text{test}$, with parameter values in $\Theta_0$ received negligible \textit{a priori} probability mass and parameter values in $\Theta_\text{test}$ distinguished from $\bm{\theta}_\star$ by a uniformly powerful test $\phi$.

We first construct $\Theta_0$, $\Theta_\text{test}$ and $\phi$, which will be used to prove \eqref{eqn:thm1} in the framework described by Lemma \ref{lem:barron}. For $\Theta_0$ and $\Theta_\text{test}$, let 
\begin{align*}
\Theta_0 &= \{(\sigma,\xi,\bm{\beta}) \in \Theta: \xi \not \in \mc{F}\} \cup \{ (\sigma,\xi,\bm{\beta}) \in \Theta: |\xi \setminus \xi_\star| > Ks\}\\
&~~\cup \left\{(\sigma,\xi,\bm{\beta}) \in \Theta: \max_{j \not \in \xi} |\beta_j| > \sigma z_{0n}\right\},\\
\Theta_\text{test} &= \Theta_1 \cup \Theta_2,
\end{align*}
with
\begin{align*}
\Theta_1 &= \left\{ (\sigma,\xi,\bm{\beta}) \in \Theta_0^c: \frac{\sigma^2}{\sigma_\star^2} \not \in \left[\frac{1-M_1\epsilon_n}{1+M_1\epsilon_n}, \frac{1+M_1\epsilon_n}{1-M_1\epsilon_n}\right]\right\},\\
\Theta_2 &= \left\{ (\sigma,\xi,\bm{\beta}) \in \Theta_0^c \cap \Theta_1^c: \left\Vert \bm{\beta} - \bm{\beta}^\star \right\Vert > M_2 \sigma_\star \epsilon_n/\sqrt{\lambda}\right\}.
\end{align*}
We take the test function as
$$\phi = \max\{\phi_1, \phi_2\},$$
with
\begin{align*}
\phi_1 &=  1\left\{\max_{\xi \in \mc{F}:~|\xi \setminus \xi_\star| \le Ks} \left|\frac{\mbf{Y}^\T(\mbf{I} - \mbf{P}_{\xi \cup \xi_\star})\mbf{Y}}{n \sigma_\star^2}-1 \right| > M_1\epsilon_n\right\},\\
\phi_2 &=  1\left\{\max_{\xi \in \mc{F}:~|\xi \setminus \xi_\star| \le Ks} \left\Vert \mbf{X}_{\xi \cup \xi_\star}^\dagger \mbf{Y} - \bm{\beta}^\star_{\xi \cup \xi_\star}\right\Vert > M_2\sigma_\star \epsilon_n/2\sqrt{\lambda}\right\}.
\end{align*}

Next, by the following lemmas, the conditions required by Lemma \ref{lem:barron} are verified with the above-defined $\Theta_0$, $\Theta_\text{test}$ and $\phi$.

\begin{lemma} \label{lem:prior}
Suppose Assumption \ref{asm1}(b)(c) hold. Then
$$\Pi(\Theta_0) \lesssim e^{-[A_2 K-o(1)] s\log p}.$$
\end{lemma}

\begin{lemma} \label{lem:test1}
$$\En \phi_1 \lesssim e^{-[M_1^2/8 - K - o(1)]n\epsilon_n^2}, ~\sup_{(\sigma, \xi, \bm{\beta}) \in \Theta_1} \Ea(1-\phi_1) \lesssim e^{-[M_1^2/8 - o(1)]n\epsilon_n^2}.$$
\end{lemma}

\begin{lemma} \label{lem:test2}
Suppose Assumption \ref{asm2} holds. Then
$$\En \phi_2 \lesssim e^{-[M_2^2/8 - K - o(1)]n\epsilon_n^2}, ~\sup_{(\sigma, \xi, \bm{\beta}) \in \Theta_2} \Ea(1-\phi_2) \lesssim e^{-[M_2^2/8 - o(1)]n\epsilon_n^2}.$$
\end{lemma}

\begin{lemma} \label{lem:merging}
Suppose Assumption \ref{asm1} holds. For any small constant $\eta > 0$,
$$\Pn \left(\int \frac{\mc{N}(\mbf{Y}|\mbf{X}\bm{\beta}, \sigma^2\mbf{I})}{\mc{N}(\mbf{Y}| \mbf{X}\bm{\beta}^\star, \sigma_\star^2\mbf{I})} d\Pi(\sigma,\xi,\bm{\beta}) \le e^{-(A_1+A_3+1-\eta)s\log p} \right) \lesssim e^{-C_\eta s\log p},$$
with constant $C_\eta > 0$ depending on $\eta$.
\end{lemma}

In particular, Lemma \ref{lem:prior} verifies condition (a) of Lemma \ref{lem:barron} with
$$\delta_0 = e^{-[A_2 K -o(1)]s\log p}.$$
Lemmas \ref{lem:test1}-\ref{lem:test2} verifies condition (b) of Lemma \ref{lem:barron} as follows:
\begin{align*}
\sup_{(\sigma, \xi, \bm{\beta}) \in \Theta_\text{test}} \Ea(1-\phi)
&\le \max_{j=1}^2 \left\{\sup_{(\sigma, \xi, \bm{\beta}) \in \Theta_j} \Ea(1-\phi_j)\right\}\\
&\lesssim e^{- [\min\{M_1,M_2\}^2/8 - o(1)]s\log p} =: \delta_1,\\
\En \phi \le \sum_{j=1}^2 \En \phi_j &\lesssim e^{- [\min\{M_1,M_2\}^2/8 - K - o(1)]s\log p} =: \delta_1'.
\end{align*}
Lemma \ref{lem:merging} verifies condition (c) of Lemma \ref{lem:barron} with
$$\delta_2 = e^{-[A_1+A_3+1-o(1)]s\log p},~~~\delta_2' = e^{-C_\eta s\log p}.$$
Finally, we note that $\delta_0 > \delta_1$, because $M_1,M_2 > \sqrt{8A_2K}$. Choosing suitable $\delta_3 = e^{-Cs\log p}$ such that
$$A_1+A_3+1-o(1) < C < A_2 K - o(1)$$
completes the proof.

\subsection{Proof of Theorem \ref{thm2}}
Since
$$\widehat{\Theta} \cap \{(\sigma, \xi, \bm{\beta}): \xi \supseteq \xi_\star\} \subseteq \widetilde{\Theta},$$
it suffices to show that
\begin{equation} \label{eqn:thm2}
\Pn \left(\Pi((\sigma, \xi, \bm{\beta}) \in \widehat{\Theta}^c \text{~or~} \xi \not \supseteq \xi_\star|\mbf{X},\mbf{Y}) \ge e^{-C_3s\log p}\right) \lesssim e^{-C_4s\log p}.
\end{equation}
To this end, we use the technique developed from Lemma \ref{lem:barron} again. Recall notation $\Theta_0$, $\Theta_1$, $\Theta_2$, $\phi_1$ and $\phi_2$ in the proof of Theorem \ref{thm1} and redefine (with a little abuse of notation)
\begin{align*}
\Theta_\text{test} &= \Theta_1 \cup \Theta_2 \cup \Theta_3, ~~\phi = \max\{\phi_1, \phi_2, \phi_3\},
\end{align*}
where
\begin{align*}
\Theta_3 &= \{(\sigma, \xi, \bm{\beta}) \in \Theta_0^c \cap \Theta_1^c: \xi \not \supseteq \xi_\star\},\\
\phi_3 &= 1\left\{ \min_{\xi \in \mathcal{F}:~|\xi \setminus \xi_\star| \le Ks,~\xi \not \supseteq \xi_\star} \left\Vert \left(\mbf{P}_{\xi \cup \xi_\star} - \mbf{P}_{\xi} \right)\mbf{Y} \right\Vert < M_3\sigma_\star\sqrt{n}\epsilon_n/2\right\}.
\end{align*}
With $M_3 = \sqrt{8A_2K}$,
$$\Theta_0 \cup \Theta_\text{test} = \widehat{\Theta}^c \cup \{(\sigma, \xi, \bm{\beta}): \xi \not \supseteq \xi_\star\}.$$

We proceed to verify conditions of Lemma \ref{lem:barron}. As in the proof of Theorem \ref{thm1}, Lemma \ref{lem:prior} and Lemma \ref{lem:merging} have verified conditions (a) and (c). It is only left to verify condition (b) for redefined $\Theta_\text{test}$ and $\phi$. The following lemma serves for this purpose.
\begin{lemma} \label{lem:test3}
Suppose Assumption \ref{asm2} and Assumption \ref{asm3} hold with $M_3 = \sqrt{8A_2K}$. Then
\begin{gather*}
\En \phi_3  \lesssim e^{-[M_3^2/8 - o(1)]n\epsilon_n^2},~\sup_{(\sigma, \xi, \bm{\beta}) \in \Theta_3} \Ea(1-\phi_3) \lesssim e^{-[M_3^2/8 - o(1)]n\epsilon_n^2}.
\end{gather*}
\end{lemma}
Next, combining Lemmas \ref{lem:test1}-\ref{lem:test2} and Lemma \ref{lem:test3} yields
\begin{align*}
\sup_{(\sigma, \xi, \bm{\beta}) \in \Theta_\text{test}} \Ea(1-\phi)
&\le \max_{j=1}^3 \left\{\sup_{(\sigma, \xi, \bm{\beta}) \in \Theta_j} \Ea(1-\phi_j)\right\}\\
&\lesssim e^{- [\min\{M_1,M_2,M_3\}^2/8 - o(1)]s\log p} =: \delta_1,\\
\En \phi \le \sum_{j=1}^3 \En \phi_j &\lesssim e^{- [\min\{M_1^2/8-K, M_2^2/8-K, M_3^2\} - o(1)]s\log p} =: \delta_1'.
\end{align*}

\subsection{Proof of Theorem \ref{thm3}}
The proof of Theorem \ref{thm1} uses two technical lemmas, which are stated as follows.

\begin{lemma} \label{lem:omega}
For any constant $\eta > 0$, let
\begin{gather*}
\Omega_1(\eta) = \cup_{t=s+1}^{(K+1)s} \cup_{\xi \in \mc{F}: ~\xi \supseteq \xi_\star,~|\xi| = t}\{\Vert (\mbf{P}_\xi - \mbf{P}_{\xi_\star})\bm{\varepsilon}\Vert^2 \ge (2+\eta)(t-s)\log p\}
\end{gather*}
and $\Omega_2 = \{\Vert \bm{\varepsilon}\Vert \ge 2\sqrt{n}\}$.  Then $\Omega(\eta) := \Omega_1(\eta) \cup \Omega_2$ satisfies
$$\P\left( \Omega(\eta) \right) \lesssim e^{-C_\eta s\log p},$$
for  some constant $C_\eta$ depending on $\eta$.
\end{lemma}

\begin{lemma} \label{lem:ratio}
Suppose Assumption \ref{asm1}(c), Assumption \ref{asm2} and condition (b) of Theorem \ref{thm3} hold. Conditional on event $\Omega(\eta)^c$ (with $\Omega(\eta)$ defined in Lemma \ref{lem:omega}),
$$\sup_{\xi \in \mc{F}:~\xi \supseteq \xi_\star, |\xi|=t}~ \frac{\Pi(\widetilde{\Theta}|\mbf{X},\mbf{Y},\xi)}{\Pi(\widetilde{\Theta}|\mbf{X},\mbf{Y},\xi_\star)} \le 2\left(\sqrt{\frac{2\pi}{n\lambda}} \times \sup_{z} h_1(z) \times p^{1+\eta}\right)^{t-s}$$
for any $t=s+1,\dots,(K+1)s$, for sufficiently large $n > N$ (where $N$ does not depends on $t$).
\end{lemma}

\begin{proof}[Proof of Theorem \ref{thm3}]
From Lemma \ref{lem:ratio} and condition (b), it follows that
\begin{align*}
& \sum_{\gamma \in \mc{F}:~\gamma \supseteq \xi_\star, |\gamma|=t} \frac{\Pi(\widetilde{\Theta} \cap \{\xi = \gamma\}|\mbf{X},\mbf{Y})}{\Pi(\widetilde{\Theta} \cap \{\xi = \xi_\star\}|\mbf{X},\mbf{Y})}\\
&= \sum_{\gamma \in \mc{F}: \gamma \supseteq \xi_\star, |\xi|=t} \frac{\pi(\gamma)}{\pi(\xi_\star)} \times \frac{\Pi(\widetilde{\Theta}|\mbf{X},\mbf{Y},\xi = \gamma)}{\Pi(\widetilde{\Theta}|\mbf{X},\mbf{Y},\xi = \xi_\star)}\\
&\lesssim \frac{\pi(|\xi|=t)}{\pi(\xi_\star)} \times \left(\sqrt{\frac{2\pi}{n\lambda}} \times \sup_{z} h_1(z) \times p^{1+\eta}\right)^{t-s}\\
&\le r_n^{t-s}
\end{align*}
Then
$$\sum_{\substack{\gamma \in \mc{F}: ~\gamma \supseteq \xi_\star,\\ s+1 \le |\gamma| \le (K+1)s}} \frac{\Pi(\widetilde{\Theta} \cap \{\xi = \gamma\}|\mbf{X},\mbf{Y})}{\Pi(\widetilde{\Theta} \cap \{\xi = \xi_\star\}|\mbf{X},\mbf{Y})} \lesssim \sum_{t=s+1}^{(K+1)s} r_n^{t-s} \le \frac{r_n}{1-r_n}.$$
Combining this result with Theorem \ref{thm2} concludes the proof.
\end{proof}

\section{Discussion}\label{sec5}

In this paper, we identify a class of generic spike-and-slab priors and then develop a unified theoretical framework to analyze these spike-and-slab methodologies. Comparing with the literature, we characterize the weakest  conditions to guarantee near optimal posterior contraction rate and consistent model selection property. Our conditions and results are general and include previous works as special cases. 

\section*{Acknowledgement}
We would like to thank Professor Edward George, Professor Faming Liang, Professor Qifan Song and Professor Jianqing Fan for helpful discussions in the initial stage of this project. 

%\begin{table}
%%\def~{\hphantom{0}}
%{
%\footnotesize
%\begin{threeparttable}
%\caption{This is a table}{%
%\begin{tabular}{llccccccccc}
%$ (n,m,J) $ & Method & $\widehat J\!<\!J $ & $\widehat J\!=\!J$ & $\widehat J\!>\!J$ & Mean & SCP 1 & SCP 2  & SCP 3 & SCP 4 \\ %\hline
%$ (100, 6, 2) $ & Proposed & 0.01 & 0.99 & 0 & 1.99(0.01) & 0.99 & 1 & NA & NA \\
%                          & Vector  & 0.32 & 0.62 & 0.06 & 1.74(0.06) & 0.67 & 0.99 & NA & NA \\
%                          & Symmetric  & 0.13 & 0.82 & 0.05 & 1.92(0.04) & 0.83 & 1 & NA & NA \\
%$ (200, 6, 2) $ & Proposed & 0 & 0.95 & 0.05 & 2.05(0.02) & 1 & 1 & NA & NA \\
%                          & Vector  & 0.64 & 0.22 & 0.14 & 1.56(0.09) & 0.17 & 0.99 & NA & NA \\
%                          & Symmetric  & 0.65 & 0.17 & 0.18 & 1.62(0.1) & 0.22 & 1 & NA & NA \\
%\end{tabular}}
%\label{table:estimation}
%\begin{tablenotes}
%\footnotesize
%\item  SCP, sure coverage probability; NA, not available.
%\end{tablenotes}
%\end{threeparttable}}
%\end{table}

\bibliographystyle{ims}
\bibliography{ref}

%\bibliographystyle{ims}
%\bibliography{misc,manifold}

\newpage
\appendix 
%Notations for Appendices
\renewcommand{\theequation}{S.\arabic{equation}}
\renewcommand{\thetable}{S.\arabic{table}}
\renewcommand{\thefigure}{S.\arabic{figure}}
\renewcommand{\thesection}{S.\arabic{section}}
\renewcommand{\thelemma}{S.\arabic{lemma}}

\vspace{30pt}
\noindent{\bf \LARGE Appendix}
\vspace{-10pt}

\section{Proofs for Example \ref{example:bernoulli} and Example \ref{example:csv}}\label{appendix:A}
We collect proofs for Example \ref{example:bernoulli} and Example \ref{example:csv} by showing both the Bernoulli prior and the CSV priors satisfy Assumption \ref{asm1} (b).  

\begin{proof}[Proof for Example \ref{example:bernoulli}]
It is elementary that
\$
\pi(\emptyset) = \left(1-\frac{1}{p}\right)^p \to e^{-1}
\$
as $p \to \infty$. Next, let $T \sim \texttt{Binomial}(p,1/p)$. Due to Lemma \ref{lem:pelekis},
\$
\sum_{\xi:|\xi|>t} \pi(\xi) &= \P( T > t) \le \P(T \ge t+1)\\
&\le \frac{(1/p)^{2(t+1)}}{2} \times \left. {p \choose t+1} \right/ {t+1 \choose t+1} \le p^{-t}.
\$
Finally,
\$
\pi(\xi_\star) &= \left(\frac{1}{p}\right)^s \left(1 - \frac{1}{p}\right)^{p-s} \gtrsim p^{-s} \times \frac{e^{-1}}{2}.
\$
\end{proof}

\begin{proof}[Proof for Example \ref{example:csv}]
Recall that both $p = p_n$ and $s = s_n$ are assumed to be sequences of $n$. First, find large integer $N_1$ such that, for any $n > N_1$, $B_2p^{-B_4} \le 1/2$. Then
\$
1 = \sum_{t=0}^p w(t) \le w(0) \left(1+ B_2p^{-B_4} + B_2^2p^{-2B_4} + \dots \right) \le \frac{w(0)}{1-B_2p^{-B_4}} \le 2 w(0)
\$
implying $w(0) \ge 1/2$. Next, find large integer $N_2$ such that, for any $n > N_2$, $\log(2B_2) < (B_4-A_2)\log p$. Then, for any $n > \max\{N_1,N_2\}$ and any $t \ge 0$,
\$
\sum_{\xi:|\xi|>t} \pi(\xi) &\le \sum_{j=t}^p w(j) \leq w(t)\left(1+ B_2p^{-B_4} + B_2^2p^{-2B_4} + \dots \right)\\
&\le \frac{w(t)}{1 - B_2p^{-B_4}} \le 2w(t) \le 2w(0) B_2^t p^{-B_4t} \le (2B_2)^t p^{-B_4t} \le p^{-A_2t}.
\$
Third, find large integer $N_3$ such that, for any $n > N_3$, $\log(B_3/2) < (A_1 - B_3 - 1)\log p$. Then, for any $n > \max\{N_1,N_3\}$,
\$
\pi(\xi_*)&= w(s) {p\choose s}^{-1} \ge w(0)B_1^s p^{-B_3s} {p\choose s}^{-1}\\
&\ge w(0) B_1^s p^{-B_3 s} p^{-s} \ge (B_1/2)^s p^{-(B_3+1)s}
\ge p^{-A_1s}
\$
Therefore, $N = \max\{N_1,N_2,N_3\}$ is the desideratum.
\end{proof}

\section{Proofs of Technical Lemmas}
%\section{Technical Lemmas and Proofs}

\begin{proof}[Proof of Lemma \ref{lem:eigenvalues}]
The first inequality is trivial. The second inequality follows from the facts that
$$\lambda_{\min}(\mbf{X}_{\xi}^\T\mbf{X}_{\xi}) = \inf_{\bm{\beta}}\left\{\frac{\bm{\beta}^\T\mbf{X}^\T\mbf{X}\bm{\beta}}{\Vert \bm{\beta}\Vert^2}: \bm{\beta}_{\xi^c} = \mbf{0}, \bm{\beta}_{\xi} \ne \mbf{0}\right\}$$
and that
$$\bm{\beta}_{\xi^c} = \mbf{0}, \bm{\beta}_{\xi} \ne \mbf{0} \Longrightarrow \bm{\beta} \ne \mbf{0}, \Vert \bm{\beta}_{\xi^c} \Vert_1 \le \alpha \Vert \bm{\beta}_{\xi}\Vert_1.$$
For the third inequality, it suffices to show upon the identifiability condition that any full-rank model $\xi$ such that $|\xi \cup \xi_\star| \le t$ results in a full-rank model $\xi \cup \xi_\star$. This is obvious for cases of underfitted models $\xi \subseteq \xi_\star$. For other cases, we prove by contradiction. Suppose for the sake of contradiction that $\xi \cup \xi_\star$ with $\xi \not \subseteq \xi_\star$ is not of full rank. We construct a vector basis of model $\xi_\star \cup \xi$ by merging all vectors of $\xi \setminus \xi_\star$ and some selected vectors of $\xi_\star$. Let $\gamma \subseteq \xi_\star$ denote the index set of the selected vectors. We must have $\gamma \subset \xi_\star$ since $\xi \cup \xi_\star$ is not of full rank. Therefore, $(\xi \setminus \xi_\star) \uplus \gamma \not \supseteq \xi_\star$, but $(\mbf{I} - \mbf{P}_{(\xi \setminus \xi_\star) \uplus \gamma})\mbf{X}_{\xi_\star} = \mbf{0}$, which contradict to the premise. 
\end{proof}

\begin{proof}[Proof of Lemma \ref{lem:prior}]
Due to Assumption \ref{asm1}(b),
\begin{align*}
\Pi(|\xi \setminus \xi_\star| > Ks)
&=\frac{\sum_{\xi \in \mathcal{F}:~|\xi \setminus \xi_\star| > Ks} \pi(\xi)}{\sum_{\xi \in \mathcal{F}} \pi(\xi)}\\
&\le \frac{\sum_{\xi:~|\xi| > Ks} \pi(\xi)}{\pi(\emptyset)}\\
&\lesssim p^{-[A_2K-o(1)]s}.
\end{align*}
Due to Assumption \ref{asm1}(c), for any $\xi \subseteq \{1,\dots,p\}$ any $\sigma > 0$,
$$\Pi\left(\left. \max_{j \not \in \xi} |\beta_j| > \sigma z_{0n} \right|\xi, \sigma\right) \le p \int 1\left\{|z| > z_{0n} \right\} h_0(z)dz \le e^{- n + \log p}.$$
Putting these pieces together completes the proof.
\end{proof}

\begin{proof}[Proof of Lemma \ref{lem:test1}, part(a)]
Write
\begin{align*}
\phi_1 &= 1\left\{ \max_{\xi \in \mc{F}:~|\xi \setminus \xi_\star| \le Ks} \left| \frac{\bm{\varepsilon}^\T (\mbf{I}-\mbf{P}_{\xi \cup \xi_\star})\bm{\varepsilon}}{n} - 1\right| > M_1\epsilon_n \right\}\\
&\le 1\left\{ \max_{\xi:~|\xi \setminus \xi_\star| \le Ks} \left| \frac{\bm{\varepsilon}^\T (\mbf{I}-\mbf{P}_{\xi \cup \xi_\star})\bm{\varepsilon}}{n} - 1\right| > M_1\epsilon_n \right\}.
\end{align*}
Since projection matrices $\mbf{P}_{\xi_1 \cup \xi_\star} \le \mbf{P}_{\xi_2 \cup \xi_\star}$ for nested models $\xi_1 \subseteq \xi_2$, the quantity $\bm{\varepsilon}^\T (\mbf{I}-\mbf{P}_{\xi \cup \xi_\star})  \bm{\varepsilon}$ achieves its minimum value at some $\xi$ of size $Ks$ and its maximum value at $\xi = \emptyset$. 
\begin{align*}
\phi_1 &\le 1\left\{ \max_{\xi:~|\xi| = Ks ~\text{or}~0} \left| \frac{\bm{\varepsilon}^\T (\mbf{I}-\mbf{P}_{\xi \cup \xi_\star}  \bm{\varepsilon})}{n} - 1\right| > M_1\epsilon_n \right\}\\
&\le \sum_{\xi: ~|\xi| = Ks ~\text{or}~0} 1\left\{ \left| \frac{\bm{\varepsilon}^\T (\mbf{I}-\mbf{P}_{\xi \cup \xi_\star}  \bm{\varepsilon})}{n} - 1\right| > M_1\epsilon_n \right\}.
\end{align*}
For each $\xi$ of size $Ks$ or $0$, write
\begin{align*}
&~~~\Pn \left( \left| \frac{\bm{\varepsilon}^\T (\mbf{I}-\mbf{P}_{\xi \cup \xi_\star}) \bm{\varepsilon}}{n} - 1\right| > M_1\epsilon_n \right)\\
&\le \P \left( \frac{\chi^2_{n - s}}{n} > 1 + M_1\epsilon_n \right) +  \P \left( \frac{\chi^2_{n - (K+1)s}}{n} < 1 - M_1\epsilon_n \right).
\end{align*}
Putting the last two displays together with the probability bound of the chi-square distribution (Lemma \ref{lem:chi2}, part (b)) yields
\begin{align*}
\En \phi_1 &\le \left[1 + {p \choose Ks} \right] \times e^{-[M_1^2/8 - o(1)]n\epsilon_n^2}\le e^{-[M_1^2/8 - K - o(1)]n\epsilon_n^2}.
\end{align*}
\end{proof}

\begin{proof}[Proof of Lemma \ref{lem:test1}, part(b)]
Define
\begin{align*}
\phi_{1,\gamma} &= 1\left\{\left|\frac{\mbf{Y}^\T(\mbf{I} - \mbf{P}_{\gamma \cup \xi_\star})\mbf{Y}}{n\sigma_\star^2} -1 \right| \le M_1\epsilon_n\right\},\\
\Theta_{1,\gamma} &= \{(\sigma, \xi, \bm{\beta}) \in \Theta_1: \xi = \gamma\}.
\end{align*}
Then
\begin{align*}
\phi_1 &= \max_{\gamma \in \mc{F}:~|\gamma \setminus \xi_\star| \le Ks} ~\phi_{1,\gamma}, ~~\Theta_1 = \cup_{\gamma \in \mc{F}: |\gamma \setminus \xi_\star| \le Ks} ~\Theta_{1,\gamma}.
\end{align*}
Applying Lemma \ref{lem:split} yields
\begin{align*}
\sup_{(\sigma, \xi, \bm{\beta}) \in \Theta_1}\Ea(1 - \phi_1)
&\le \max_{\gamma  \in \mc{F}:~|\gamma\setminus\xi_\star| \le Ks} ~\sup_{(\sigma, \xi, \bm{\beta}) \in \Theta_{1,\gamma}} \Ea (1-\phi_{1,\gamma})\\
&= \sup_{(\sigma, \xi, \bm{\beta}) \in \Theta_1} \Ea (1-\phi_{1,\xi}). 
\end{align*}
We proceed to bound, for any $(\sigma, \xi, \bm{\beta}) \in \Theta_1$, 
$$\Ea (1-\phi_{1,\xi}) = \Pa \left(\left|\frac{\mbf{Y}^\T(\mbf{I} - \mbf{P}_{\xi \cup \xi_\star})\mbf{Y}}{n\sigma_\star^2} -1 \right| \le M_1\epsilon_n\right).$$
To this end, the restriction $\frac{\sigma^2}{\sigma_\star^2} \not \in \left[\frac{1-M_1\epsilon_n}{1+M_1\epsilon_n}, \frac{1+M_1\epsilon_n}{1-M_1\epsilon_n}\right]$ of $\Theta_1 \subseteq \Theta_0^c$ implies that
$$\left|\frac{\mbf{Y}^\T(\mbf{I} - \mbf{P}_{\xi \cup \xi_\star})\mbf{Y}}{n\sigma_\star^2} -1 \right| \le M_1\epsilon_n \Longrightarrow \left|\frac{\mbf{Y}^\T(\mbf{I} - \mbf{P}_{\xi \cup \xi_\star})\mbf{Y}}{n\sigma^2} - 1\right| > M_1\epsilon_n.$$
Another restriction $\max_{j \not \in \xi} |\beta_j| \le \sigma z_{0n}$ of $\Theta_1 \subseteq \Theta_0^c$ implies that $\Vert \mbf{X}_{\xi^c}\bm{\beta}_{\xi^c}\Vert \le  \sigma \sqrt{n} p z_{0n}$. As we will show later, under $\Pa$,
\begin{equation} \label{todo}
\left|\frac{\mbf{Y}^\T(\mbf{I} - \mbf{P}_{\xi \cup \xi_\star})\mbf{Y}}{n\sigma^2} - 1\right| > M_1\epsilon_n \Longrightarrow \left| \frac{\bm{\varepsilon}^\T (\mbf{I} - \mbf{P}_{\xi \cup \xi_\star}) \bm{\varepsilon}}{n} - 1 \right|> M_1\epsilon_n - 3pz_{0n}.
\end{equation}
Then
\begin{align*}
&~~~\Pa (1-\phi_{1,\xi}) \le \P\left( \left| \frac{\bm{\varepsilon}^\T (\mbf{I} - \mbf{P}_{\xi \cup \xi_\star}) \bm{\varepsilon}}{n} - 1 \right|> (M_1\epsilon_n-3pz_{0n})\right)\\
& \le  \P \left( \frac{\chi^2_{n-s}}{n} > 1 + (M_1\epsilon_n-3pz_{0n})\right) + \P \left( \frac{\chi^2_{n-(K+1)s}}{n} < 1 -(M_1\epsilon_n-3pz_{0n})\right).
\end{align*}
Note that this bound holds uniformly for all $(\sigma, \xi, \bm{\beta}) \in \Theta_1$, and that Assumption \ref{asm1}(c) derives that $pz_{0n}/\epsilon_n \to 0$. Therefore, the probability bound of the chi-square distribution (Lemma \ref{lem:chi2}, part (a)) concludes the proof.

It is only left to show \eqref{todo}. For simplicity of notation, we write $\mbf{R} = \mbf{I} - \mbf{P}_{\xi \cup \xi_\star}$, $\mbf{b} = \mbf{X}_{\xi^c}\bm{\beta}_{\xi^c}/\sigma\sqrt{n}$, then $\mbf{Y}^\T(\mbf{I} - \mbf{P}_{\xi \cup \xi_\star})\mbf{Y}/n\sigma^2 = \Vert \mbf{R}(\bm{\varepsilon} + \mbf{b})\Vert^2$ under $\Pa$. Since
$$\Vert \mbf{R}\bm{\varepsilon}\Vert^2 - 2 \Vert \mbf{R}\bm{\varepsilon}\Vert \Vert \mbf{b} \Vert \le \Vert \mbf{R}(\bm{\varepsilon}+\mbf{b}) \Vert^2 \le (\Vert \mbf{R}\bm{\varepsilon}\Vert + \Vert \mbf{b} \Vert)^2,$$
we have for small $\Vert \mbf{b}\Vert$,
\begin{align*}
\Vert \mbf{R}(\bm{\varepsilon}+\mbf{b}) \Vert^2  > 1 + M_1\epsilon_n
&\Longrightarrow \Vert \mbf{R}\bm{\varepsilon}\Vert^2 > \left(\sqrt{1 + M_1\epsilon_n} - \Vert \mbf{b} \Vert\right)^2\\
&\Longrightarrow \Vert \mbf{R}\bm{\varepsilon}\Vert^2 > 1 + M_1\epsilon_n - 3\Vert \mbf{b} \Vert
\end{align*}
and
\begin{align*}
\Vert \mbf{R}(\bm{\varepsilon}+\mbf{b}) \Vert^2  < 1 - M_1\epsilon_n
&\Longrightarrow \Vert \mbf{R}\bm{\varepsilon}\Vert^2 < \left(\Vert \mbf{b} \Vert + \sqrt{\Vert \mbf{b} \Vert^2 + 1 - M_1\epsilon_n}\right)^2\\
&\Longrightarrow \Vert \mbf{R}\bm{\varepsilon}\Vert^2 < 1 - M_1\epsilon_n + 3\Vert \mbf{b} \Vert.
\end{align*}
\end{proof}

\begin{proof}[Proof of Lemma \ref{lem:test2}, part (a)]
For any $\xi \in \mc{F}$ such that $|\xi \setminus \xi_\star| \le Ks$, write
\begin{align*}
\frac{\Vert \mbf{X}_{\xi \cup \xi_\star}^\dagger \mbf{Y} - \bm{\beta}^\star_{\xi \cup \xi_\star} \Vert^2}{\sigma_\star^2}
&= \bm{\varepsilon}^\T \mbf{X}_{\xi \cup \xi_\star} (\mbf{X}_{\xi \cup \xi_\star}^\T \mbf{X}_{\xi \cup \xi_\star})^{-2}  \mbf{X}_{\xi \cup \xi_\star}^\T \bm{\varepsilon}\\
&\le \frac{\bm{\varepsilon}^\T \mbf{X}_{\xi \cup \xi_\star} (\mbf{X}_{\xi \cup \xi_\star}^\T \mbf{X}_{\xi \cup \xi_\star})^{-1}  \mbf{X}_{\xi \cup \xi_\star}^\T \bm{\varepsilon}}{\lambda_{\min}(\mbf{X}_{\xi\cup\xi^\star}^\T\mbf{X}_{\xi\cup\xi_\star})} \le \frac{\bm{\varepsilon}^\T \mbf{P}_{\xi \cup \xi_\star}  \bm{\varepsilon}}{n\lambda},
\end{align*}
implying
\begin{align*}
\phi_2 &\le 1\left\{ \max_{\xi \in \mc{F}:~|\xi\setminus\xi_\star| \le Ks} \bm{\varepsilon}^\T \mbf{P}_{\xi \cup \xi_\star}  \bm{\varepsilon} > M_2^2n\epsilon_n^2/4 \right\}\\
&\le 1\left\{ \max_{\xi:~|\xi\setminus\xi_\star| \le Ks} \bm{\varepsilon}^\T \mbf{P}_{\xi \cup \xi_\star}  \bm{\varepsilon} > M_2^2n\epsilon_n^2/4 \right\}.
\end{align*}
Since projection matrices $\mbf{P}_{\xi_1 \cup \xi_\star} \le \mbf{P}_{\xi_2 \cup \xi_\star}$ for nested models $\xi_1 \subseteq \xi_2$, the quantity $\bm{\varepsilon}^\T \mbf{P}_{\xi \cup \xi_\star}  \bm{\varepsilon}$ achieves its maximum value at some $\xi$ of size $Ks$.
\begin{align*}
\phi_2 \le 1\left\{ \max_{\xi:~|\xi| = Ks} \bm{\varepsilon}^\T \mbf{P}_{\xi \cup \xi_\star}  \bm{\varepsilon} > M_2^2n\epsilon_n^2/4 \right\}\\
\le \sum_{\xi:~|\xi| = Ks} 1\left\{\bm{\varepsilon}^\T \mbf{P}_{\xi \cup \xi_\star}  \bm{\varepsilon} > M_2^2n\epsilon_n^2/4 \right\}.
\end{align*}
For each $\xi$ of size $Ks$, we note that $\rank{\xi \cup \xi_\star} \le (K+1)s$ and write
\begin{align*}
\Pn\left(\bm{\varepsilon}^\T \mbf{P}_{\xi \cup \xi_\star}  \bm{\varepsilon} > M_2^2n\epsilon_n^2/4 \right)
&\le \P\left(\chi^2_{(K+1)s} > M_2^2n\epsilon_n^2/4 \right).
\end{align*}
Putting the last two displays together with the probability bound of the chi-square distribution (Lemma \ref{lem:chi2}, part (b)) yields
\begin{align*}
\En \phi_2 &\le {p \choose Ks} \times e^{-[M_2^2/8 - o(1)]n\epsilon_n^2}\le e^{-[M_2^2/8 - K - o(1)]n\epsilon_n^2}.
\end{align*}
\end{proof}

\begin{proof}[Proof of Lemma \ref{lem:test2}, part(b)]
Define
\begin{align*}
\phi_{2,\gamma} &= 1\left\{ \Vert \mbf{X}_{\gamma \cup \xi_\star}^\dagger \mbf{Y} - \bm{\beta}^\star_{\gamma \cup \xi_\star} \Vert  > \frac{M_2 \sigma_\star\epsilon_n}{2\sqrt{\lambda}}\right\},\\
\Theta_{2,\gamma} &= \{(\sigma, \xi, \bm{\beta}) \in \Theta_2: \xi = \gamma\},
\end{align*}
then
\begin{align*}
\phi_2 &= \max_{\gamma  \in \mc{F}:~|\gamma\setminus\xi_\star| \le Ks}~\phi_{2,\gamma}\\
\Theta_2 &= \cup_{\gamma  \in \mc{F}:~|\gamma\setminus\xi_\star| \le Ks}~\Theta_{2,\gamma}.
\end{align*}
Applying Lemma \ref{lem:split} yields
\begin{align*}
\sup_{(\sigma, \xi, \bm{\beta}) \in \Theta_2}\Ea(1 - \phi_2)
&\le \max_{\gamma  \in \mc{F}:~|\gamma\setminus\xi_\star| \le Ks} ~\Ea (1-\phi_{2,\gamma})\\
&= \sup_{(\sigma, \xi, \bm{\beta}) \in \Theta_2} ~\Ea (1-\phi_{2,\xi}). 
\end{align*}
We proceed to bound, for any $(\sigma, \xi, \bm{\beta}) \in \Theta_2$, 
\begin{align*}
\Ea (1-\phi_{2,\xi}) &= \Pa \left(\Vert \mbf{X}_{\xi \cup \xi_\star}^\dagger \mbf{Y} - \bm{\beta}^\star_{\xi \cup \xi_\star} \Vert  \le \frac{M_2 \sigma_\star\epsilon_n}{2\sqrt{\lambda}}\right). 
\end{align*}
Under $\Pa$, restrictions $\frac{\sigma^2}{\sigma_\star^2} \not \in \left[\frac{1-M_1\epsilon_n}{1+M_1\epsilon_n}, \frac{1+M_1\epsilon_n}{1-M_1\epsilon_n}\right]$ and $\max_{j \not \in \xi} |\beta_j| \le \sigma z_{0n}$ of $\Theta_1 \subseteq \Theta_0^c$ imply 
\begin{align*}
&~~~\Vert \mbf{X}_{\xi \cup \xi_\star}^\dagger \mbf{Y} - \bm{\beta}^\star_{\xi \cup \xi_\star} \Vert  \le \frac{M_2 \sigma_\star\epsilon_n}{2\sqrt{\lambda}}\\
&\Longrightarrow \Vert \bm{\beta}_{\xi \cup \xi_\star} - \bm{\beta}^\star_{\xi \cup \xi_\star} + \mbf{X}_{\xi \cup \xi_\star}^\dagger ( \mbf{X}_{\xi^c \cap \xi_\star^c} \bm{\beta}_{\xi^c \cap \xi_\star^c} +\sigma\bm{\varepsilon}) \Vert  \le \frac{M_2 \sigma_\star\epsilon_n}{2\sqrt{\lambda}}\\
&\Longrightarrow \Vert \mbf{X}_{\xi \cup \xi_\star}^\dagger ( \mbf{X}_{\xi^c \cap \xi_\star^c} \bm{\beta}_{\xi^c \cap \xi_\star^c} +\sigma\bm{\varepsilon}) \Vert  \ge \Vert \bm{\beta} - \bm{\beta}^\star \Vert -  \Vert \bm{\beta}_{\xi^c} \Vert - \frac{M_2 \sigma_\star\epsilon_n}{2\sqrt{\lambda}}\\
&\Longrightarrow \frac{1}{\sqrt{n\lambda}}\Vert \mbf{P}_{\xi \cup \xi_\star} ( \mbf{X}_{\xi^c \cap \xi_\star^c} \bm{\beta}_{\xi^c \cap \xi_\star^c} +\sigma\bm{\varepsilon}) \Vert \ge \Vert \bm{\beta}  - \bm{\beta}^\star \Vert - \Vert \bm{\beta}_{\xi^c} \Vert - \frac{M_2 \sigma_\star\epsilon_n}{2\sqrt{\lambda}}\\
&\Longrightarrow \frac{1}{\sqrt{n\lambda}} \left(\sqrt{np} \Vert \bm{\beta}_{\xi^c}\Vert + \sigma \Vert \mbf{P}_{\xi \cup \xi_\star} \bm{\varepsilon}\Vert \right) \ge \Vert \bm{\beta} - \bm{\beta}^\star \Vert - \Vert \bm{\beta}_{\xi^c} \Vert - \frac{M_2 \sigma_\star\epsilon_n}{2\sqrt{\lambda}}\\
&\Longrightarrow \Vert \mbf{P}_{\xi \cup \xi_\star} \bm{\varepsilon} \Vert \ge \left(\frac{M_2}{2} \sqrt{\frac{1-M_1\epsilon_n}{1+M_1\epsilon_n}}  - \frac{pz_{0n}}{\epsilon_n} - \frac{pz_{0n}\sqrt{\lambda}}{\epsilon_n\sqrt{p}}\right) \sqrt{n}\epsilon_n\\
&\Longrightarrow \Vert \mbf{P}_{\xi \cup \xi_\star} \bm{\varepsilon} \Vert \ge \left(\frac{M_2}{2} \sqrt{\frac{1-M_1\epsilon_n}{1+M_1\epsilon_n}}  - \frac{2pz_{0n}}{\epsilon_n}\right)\sqrt{n}\epsilon_n.
\end{align*}
Thus,
$$\Pa (1-\phi_{2,\xi}) \le \P \left( \chi^2_{(K+1)s} \ge \left(\frac{M_2}{2} \sqrt{\frac{1-M_1\epsilon_n}{1+M_1\epsilon_n}}  - \frac{2pz_{0n}}{\epsilon_n}\right)^2 n\epsilon_n^2 \right).$$
Note that this bound holds uniformly for all $(\sigma, \xi, \bm{\beta}) \in \Theta_2$. Assumption \ref{asm1}(c) deriving that $pz_{0n}/\epsilon_n \to 0$ and the probability bound of the chi-square distribution (Lemma \ref{lem:chi2}, part (b)) conclude the proof.
\end{proof}

\begin{proof}[Proof of Lemma \ref{lem:merging}]
The proof consists of four steps.
\begin{enumerate}[label=(\roman*)]
\item Let
\begin{equation*}
\Theta_\star = \Theta_\star(\eta_1, \eta_2) := \left\{(\sigma, \xi, \bm{\beta}):
\begin{split}
& \sigma^2/\sigma_\star^2 \in [1, 1+\eta_1\epsilon_n^2],\\
&\xi = \xi_\star,\\
& |\beta_j| \le \sigma z_{0n}, j \not \in \xi_\star,\\
& |\beta_j - \beta^\star_j | \le \eta_2 \sigma \epsilon_n/s, j \in \xi_\star
\end{split}
\right\}.
\end{equation*}
It is obvious that
$$\int \frac{\mc{N}(\mbf{Y}| \mbf{X}\bm{\beta}, \sigma^2\mbf{I})}{\mc{N}(\mbf{Y}| \mbf{X}\bm{\beta}^\star, \sigma_\star^2\mbf{I})} d\Pi \ge \Pi(\Theta_\star) \inf_{ (\sigma, \xi, \bm{\beta}) \in \Theta_\star} \frac{\mc{N}(\mbf{Y}| \mbf{X}\bm{\beta}, \sigma^2\mbf{I})}{\mc{N}(\mbf{Y}| \mbf{X}\bm{\beta}^\star, \sigma_\star^2\mbf{I})},$$
\item Prove that if $\eta_2 < 1$ then
$$\Pi(\Theta_\star) \gtrsim e^{-[A_1+A_3+1-o(1)]s\log p/2}.$$
\item Prove the contrapositive of the implication statement
\begin{align*}
\inf_{ (\sigma, \xi, \bm{\beta}) \in \Theta_\star} & \frac{\mc{N}(\mbf{Y}| \mbf{X}\bm{\beta}, \sigma^2\mbf{I})}{\mc{N}(\mbf{Y}| \mbf{X}\bm{\beta}^\star, \sigma_\star^2\mbf{I})} \le e^{-[\eta_1 + \eta_2^2/2 + \sqrt{C}\eta_2+\eta_3+o(1)]s\log p}\\
&\Longrightarrow \Vert \mbf{P}_{\xi_\star} \bm{\varepsilon}\Vert^2 \ge C n\epsilon_n^2~\text{or}~\Vert \bm{\varepsilon}\Vert^2 \ge 4n,
\end{align*}
for any small $\eta_3 >0$.
\item Prove $$\Pn \left(\Vert \mbf{X}_{\xi_\star} \bm{\varepsilon}\Vert^2 \ge C n\epsilon_n^2 ~\text{or}~ \Vert \bm{\varepsilon}\Vert^2 \ge 4n \right) \lesssim e^{-[C/2-o(1)]n\epsilon_n^2}.$$
\end{enumerate}
Setting sufficiently small $\eta_1,\eta_2,\eta_3$ and suitable $C$ such that
$$\eta > \eta_1 + \eta_2^2/2 + \sqrt{C}\eta_2 + \eta_3, ~~~ C_\eta > C/2$$
completes the proof.

\noindent \textit{Proof of (ii):}
\begin{align*}
\Pi(\Theta_\star)
&= \int_{\sigma_\star^2}^{(1+\eta_1\epsilon_n^2)\sigma_\star^2} g(\sigma^2)d\sigma^2 \times \pi(\xi_\star)\\
&~~ \times \Pi\left(\left. \max_{j \not \in \xi} |\beta_j| \le \sigma z_{0n} \right| \xi = \xi_\star, \sigma\right) \\
&~~ \times \prod_{j \in \xi_\star} \int_{\beta^\star_j - \eta_2 \sigma \epsilon_n/s}^{\beta^\star_j + \eta_2 \sigma \epsilon_n/s} h_1\left(\frac{\beta_j}{\sigma}\right)d\left(\frac{\beta_j}{\sigma}\right). 
\end{align*}
The first term, due to Assumption \ref{asm1}(a), is bounded from below as
$$\int_{\sigma_\star^2}^{\sigma_\star^2(1+\eta_1 \epsilon_n^2)} g(\sigma^2)d \sigma^2 \ge \eta_1 \sigma_\star^2 \epsilon_n^2 g(\sigma_\star^2)/2 \gtrsim e^{-c s\log p},$$
for any small $c > 0$. The second term, due to Assumption \ref{asm1}(b), is bounded from below as
$$\pi(\xi_\star) \succeq p^{-A_1 s}.$$
The third term, due to Assumption \ref{asm1}(c), is bounded from below as
$$\Pi\left(\left. \max_{j \in \xi^c} |\beta_j| \le \sigma z_{0n} \right| \xi = \xi_\star, \sigma\right) \ge 1 -e^{-n + \log p}.$$
The fourth term, due to Assumption \ref{asm1}(d), is bounded from below as
\begin{align*}
&~~~\prod_{j \in \xi_\star} \int_{\beta^\star_j / \sigma - \eta_2 \epsilon_n / s}^{\beta^\star_j / \sigma + \eta_2 \epsilon_n / s} h_1(z)dz\\
&\ge \left(2\eta_2 \sqrt{\frac{\log p}{ns}} \times \inf\left\{h_1(z): |z| \le \max_{j \in \xi_\star} |\beta^\star_j/\sigma| + \eta_2 \epsilon_n / s \right\}\right)^s\\
&\ge \left(2\eta_2 \sqrt{\frac{\log p}{ns}} \times \inf\left\{h_1(z): |z| \le z_{1n}\right\}\right)^s \gtrsim p^{-(A_3 + 1)s}.
\end{align*}

\noindent \textit{Proof of (iii):} we are going to show that, given $\Vert \mbf{P}_{\xi_\star} \bm{\varepsilon}\Vert^2 < C n\epsilon_n^2$ and $\Vert \bm{\varepsilon}\Vert^2 < 4n$, the density ratio
$$\Lambda = \frac{\mc{N}(\mbf{Y}| \mbf{X}\bm{\beta}, \sigma^2\mbf{I})}{\mc{N}(\mbf{Y}| \mbf{X}\bm{\beta}^\star, \sigma_\star^2\mbf{I})} \gtrsim e^{-(\eta_1 + \eta_2^2/2 + \sqrt{C}\eta_2 +\eta_3)s\log p}$$
for any $(\sigma, \xi, \bm{\beta}) \in \Theta_\star$ and any samll constant $\eta_3 > 0$. Write
\begin{align*}
-2\log \Lambda &= \Vert \sigma_\star \bm{\varepsilon} + \mbf{X}(\bm{\beta}^\star - \bm{\beta})\Vert^2/\sigma^2 - \Vert \bm{\varepsilon} \Vert^2 + 2n\log(\sigma^2/\sigma_\star^2)\\
&= (\sigma_\star^2/\sigma^2-1) \Vert\bm{\varepsilon}\Vert^2 + \Vert \mbf{X}(\bm{\beta}^\star - \bm{\beta})/\sigma\Vert^2 + 2\sigma_\star \bm{\varepsilon}^\T \mbf{X}_{\xi_\star}(\bm{\beta}_{\xi_\star} - \bm{\beta})/\sigma^2\\
&~~+ 2\sigma_\star \bm{\varepsilon}^\T \mbf{X}_{\xi_\star^c}\bm{\beta}_{\xi_\star^c}/\sigma^2 + 2n\log(\sigma^2/\sigma_\star^2)\\
&\le \Vert \mbf{X}(\bm{\beta}^\star - \bm{\beta})/\sigma\Vert^2 + 2\Vert \mbf{X}_{\xi_\star}^\T \bm{\varepsilon} \Vert \Vert (\bm{\beta}_{\xi_\star} - \bm{\beta}_{\xi_\star})/\sigma \Vert\\
&~~+ 2\Vert  \bm{\varepsilon}\Vert \Vert  \mbf{X}_{\xi_\star^c}\bm{\beta}_{\xi_\star^c}/\sigma\Vert + 2n\log(\sigma^2/\sigma_\star^2). 
\end{align*}
where the definition of $\Theta_\star$ enforces restrictions that
\begin{align*}
\log(\sigma^2/\sigma_\star^2) &\le \eta_1 \epsilon_n^2,~~~ 
\Vert \bm{\beta}^\star_{\xi_\star} - \bm{\beta}_{\xi_\star}\Vert
\le \eta_2 \sigma \epsilon_n/\sqrt{s},~~~ 
\Vert \mbf{X}_{\xi_\star^c}\bm{\beta}_{\xi_\star^c}\Vert
\le \sigma \sqrt{n}pz_{0n},\\
\Vert \mbf{X}(\bm{\beta}^\star - \bm{\beta})\Vert
&\le \Vert \mbf{X}_{\xi_\star^c} \bm{\beta}_{\xi_\star^c}\Vert + \Vert \mbf{X}_{\xi_\star}(\bm{\beta}^\star_{\xi_\star} - \bm{\beta}_{\xi_\star})\Vert \le \sigma \sqrt{n}pz_{0n} + \sqrt{ns} \times \frac{\eta_2 \sigma \epsilon_n}{\sqrt{s}},
\end{align*}
and the premise derives that $\Vert \bm{\varepsilon} \Vert < 2\sqrt{n}$ and that
$$\Vert \mbf{X}_{\xi_\star}^\T \bm{\varepsilon} \Vert^2 \le \lambda_{\max}(\mbf{X}_{\xi_\star}^\T\mbf{X}_{\xi_\star})\Vert \mbf{P}_{\xi_\star} \bm{\varepsilon} \Vert^2 < ns \times Cn\epsilon_n^2.$$
Collecting these pieces together yields the desired lower bound for the density ratio.

\noindent \textit{Proof of (iv):} it follows from the facts that $\Vert \mbf{P}_{\xi_\star} \bm{\varepsilon}\Vert^2 \sim \chi_s^2$, $\Vert \bm{\varepsilon}\Vert^2 \sim \chi_n^2$ and the probability bounds of the chi-squared distribution (Lemma \ref{lem:chi2}).
\end{proof}

\begin{proof}[Proof of Lemma \ref{lem:test3}, part (a)]
We first show that
\begin{equation} \label{schur}
\min_{\xi \not \supseteq \xi_\star: |\xi\setminus\xi_\star| \le Ks} \left\Vert \left(\mbf{P}_{\xi \cup \xi_\star} - \mbf{P}_{\xi} \right)\mbf{X}_{\xi_\star} \bm{\beta}^\star_{\xi_\star}\right\Vert \ge M_3\sigma_\star \sqrt{n}\epsilon_n.
\end{equation}
Indeed, for any $\xi \not \supseteq \xi_\star$ with set difference $|\xi\setminus\xi_\star| \le Ks$,
\begin{align*}
\left\Vert \left(\mbf{P}_{\xi \cup \xi_\star} - \mbf{P}_{\xi}\right)\mbf{X}_{\xi_\star} \bm{\beta}^\star_{\xi_\star}\right\Vert^2
&= \left\Vert \left(\mbf{I} - \mbf{P}_{\xi} \right)\mbf{X}_{\xi_\star} \bm{\beta}^\star_{\xi_\star}\right\Vert^2\\
&= \left\Vert \left(\mbf{I} - \mbf{P}_{\xi} \right)\mbf{X}_{\xi_\star \setminus \xi} \bm{\beta}^\star_{\xi_\star \setminus \xi}\right\Vert^2\\
&= \bm{\beta}_{\xi_\star \setminus \xi}^\T\mbf{X}_{\xi_\star \setminus \xi}^\T \left(\mbf{I} - \mbf{P}_{\xi} \right) \mbf{X}_{\xi_\star \setminus \xi} \bm{\beta}_{\xi_\star \setminus \xi}
\end{align*}
Note that $\mbf{X}_{\xi_\star \setminus \xi}^\T \left(\mbf{I} - \mbf{P}_{\xi} \right) \mbf{X}_{\xi_\star \setminus \xi}$ is the Schur complement of the principal submatrix $\mbf{X}_{\xi_\star \setminus \xi}^\T\mbf{X}_{\xi_\star \setminus \xi}$ in the matrix $\mbf{X}_{\xi \cup \xi_\star}^\T\mbf{X}_{\xi \cup \xi_\star}$. Thus, by Lemma \ref{lem:schur},
\begin{equation*}
\lambda_{\min}\left(\mbf{X}_{\xi_\star \setminus \xi}^\T \left(\mbf{I} - \mbf{P}_{\xi} \right) \mbf{X}_{\xi_\star \setminus \xi}\right) \ge \lambda_{\min}\left(\mbf{X}_{\xi \cup \xi_\star}^\T\mbf{X}_{\xi \cup \xi_\star}\right) \ge n\lambda.   
\end{equation*}
It further leads to \eqref{schur} as
\begin{align*}
\left\Vert \left(\mbf{P}_{\xi \cup \xi_\star}  - \mbf{P}_{\xi} \right)\mbf{X}_{\xi_\star} \bm{\beta}^\star_{\xi_\star}\right\Vert^2
\ge n\lambda \Vert \bm{\beta}^\star_{\xi_\star \setminus \xi} \Vert^2
\ge n\lambda |\xi_\star \setminus \xi| \min_{j \in \xi_\star} |\beta^\star_j|^2
\ge M_3^2 \sigma_\star^2 n \epsilon_n^2.
\end{align*}
From the fact that $\mbf{P}_{\xi \cup \xi_\star} \le \mbf{P}_{\xi} + \mbf{P}_{\xi_\star}$, it follows that
\begin{align*}
\phi_3 &= 1\left\{ \min_{\xi \not \supseteq \xi_\star: |\xi\setminus\xi_\star| \le Ks} \left\Vert \left(\mbf{P}_{\xi \cup \xi_\star} - \mbf{P}_{\xi} \right)\left(\mbf{X}_{\xi_\star}\bm{\beta}^\star_{\xi_\star} + \sigma_\star \bm{\varepsilon}\right) \right\Vert < M_3\sigma_\star\sqrt{n}\epsilon_n/2\right\}\\
&\le 1\left\{\min_{\xi \not \supseteq \xi_\star: |\xi\setminus\xi_\star| \le Ks} \left\Vert \left(\mbf{P}_{\xi \cup \xi_\star} - \mbf{P}_{\xi} \right) \mbf{X}_{\xi_\star} \bm{\beta}^\star_{\xi_\star} \right\Vert \right.\\
&~~~~~~~~~\left.- \max_{\xi \not \supseteq \xi_\star: |\xi\setminus\xi_\star| \le Ks} \left\Vert \left(\mbf{P}_{\xi \cup \xi_\star}  - \mbf{P}_{\xi} \right) \sigma_\star \bm{\varepsilon} \right\Vert < M_3\sigma_\star \sqrt{n}\epsilon_n/2\right\}\\
&\le 1\left\{\max_{\xi \not \supseteq \xi_\star: |\xi\setminus\xi_\star| \le Ks} \left\Vert \left(\mbf{P}_{\xi \cup \xi_\star}  - \mbf{P}_{\xi} \right) \sigma_\star \bm{\varepsilon} \right\Vert > M_3\sigma_\star \sqrt{n}\epsilon_n/2\right\}\\
&\le 1\left\{ \left\Vert \mbf{P}_{\xi_\star} \bm{\varepsilon} \right\Vert > M_3 \sqrt{n}\epsilon_n/2\right\}.
\end{align*}
Putting it together with the tail probability bound of the chi-square distribution (Lemma \ref{lem:chi2}) yields
$$\En \phi_3 \le \P\left(\chi^2_s > M_3 n\epsilon_n^2/4\right) \le  e^{-[M_3^2/8 - o(1)] n\epsilon_n^2}.$$
\end{proof}

\begin{proof}[Proof of Lemma \ref{lem:test3}, part(b)]
Define
\begin{align*}
\phi_{3,\gamma} &= 1\left\{ \left\Vert \left(\mbf{P}_{\gamma \cup \xi_\star} - \mbf{P}_{\gamma} \right) \mbf{Y} \right\Vert  < M_3\sigma_\star \sqrt{n}\epsilon_n/2\right\},\\
\Theta_{3,\gamma} &= \{(\sigma, \xi, \bm{\beta}) \in \Theta_3: \xi = \gamma\}.
\end{align*}
Then
\begin{align*}
\phi_3 &= \max_{\gamma \not \supseteq \xi_\star: |\gamma \setminus \xi_\star| \le Ks} \phi_{3,\gamma}, ~
\Theta_3 = \cup_{\gamma \not \supseteq \xi_\star: |\gamma \setminus \xi_\star| \le Ks} \Theta_{3,\gamma}.
\end{align*}
Applying Lemma \ref{lem:split} yields
\begin{align*}
\sup_{(\sigma, \xi, \bm{\beta}) \in \Theta_3}\Ea(1 - \phi_3)
&\le \max_{\gamma \not \supseteq \xi_\star: |\gamma \setminus \xi_\star| \le Ks} ~\sup_{(\sigma, \xi, \bm{\beta}) \in \Theta_{3,\gamma}} \Ea (1-\phi_{3,\gamma})\\
&= \sup_{(\sigma, \xi, \bm{\beta}) \in \Theta_3} \Ea (1-\phi_{3,\xi}). 
\end{align*}
We proceed to bound, for any $(\sigma, \xi, \bm{\beta}) \in \Theta_3$,
\begin{align*}
\Ea (1-\phi_{3,\xi}) &= \Pa \left(\left\Vert \left(\mbf{P}_{\xi \cup \xi_\star} - \mbf{P}_{\xi} \right) \mbf{Y} \right\Vert  \ge M_3\sigma_\star \sqrt{n}\epsilon_n/2\right). 
\end{align*}
Under $\Pa$, restrictions $\frac{\sigma^2}{\sigma_\star^2} \not \in \left[\frac{1-M_1\epsilon_n}{1+M_1\epsilon_n}, \frac{1+M_1\epsilon_n}{1-M_1\epsilon_n}\right]$ and $\max_{j \not \in \xi} |\beta_j| \le \sigma z_{0n}$ of $\Theta_1 \subseteq \Theta_0^c$ and the fact that $\mbf{P}_{\xi \cup \xi_\star} \le \mbf{P}_{\xi} + \mbf{P}_{\xi_\star}$ imply that
\begin{align*}
&~~~\left\Vert \left(\mbf{P}_{\xi \cup \xi_\star} - \mbf{P}_{\xi} \right) \mbf{Y} \right\Vert  \ge M_3\sigma_\star \sqrt{n}\epsilon_n/2\\
&\Longrightarrow \left\Vert \left(\mbf{P}_{\xi \cup \xi_\star} - \mbf{P}_{\xi} \right) (\mbf{X}_{\xi^c} \bm{\beta}_{\xi^c } + \sigma \bm{\varepsilon}) \right\Vert  \ge M_3\sigma_\star \sqrt{n}\epsilon_n/2\\
&\Longrightarrow \Vert \mbf{X}_{\xi^c} \bm{\beta}_{\xi^c } \Vert + \sigma \Vert P_{\xi_\star} \bm{\varepsilon} \Vert  \ge M_3\sigma_\star \sqrt{n}\epsilon_n/2\\
&\Longrightarrow  \sigma \epsilon_n + \sigma \Vert P_{\xi_\star} \bm{\varepsilon} \Vert  \ge M_3\sigma_\star \sqrt{n}\epsilon_n/2\\
&\Longrightarrow\Vert P_{\xi_\star} \bm{\varepsilon} \Vert  \ge \left(\frac{M}{2}\sqrt{\frac{1-M_1\epsilon_n}{1+M_1\epsilon_n}} -\frac{pz_{0n}}{\epsilon_n} \right) \sqrt{n}\epsilon_n.
\end{align*}
Thus,
$$\Pa (1-\phi_{3,\xi}) \le  \P \left( \chi^2_s \ge \left(\frac{M_3}{2}\sqrt{\frac{1-M_1\epsilon_n}{1+M_1\epsilon_n}} -\frac{pz_{0n}}{\epsilon_n} \right)^2 n\epsilon_n^2 \right).$$
Note that this bound holds uniformly for all $(\sigma, \xi, \bm{\beta}) \in \Theta_3$ and that Assumption \ref{asm1}(c) derives that $pz_{0n}/\epsilon_n \to 0$. Therefore, the probability bound of the chi-square distribution (Lemma \ref{lem:chi2}, part (b)) concludes the proof.
\end{proof}

\begin{proof}[Proof of Lemma \ref{lem:omega}]
First, for any full-rank overfitted model $\xi$ of size $t \le (K+1)s$,
\begin{align*}
\P\left(\Vert (\mbf{P}_\xi - \mbf{P}_{\xi_\star})\bm{\varepsilon}\Vert^2
\ge 2(1+\eta)(t-s)\log p\right)
&\le \P\left(\chi_{t-s}^2 \ge 2(1+\eta)(t-s)\log p\right)\\
&\lesssim p^{-(1+\eta/2)(t-s)},
\end{align*}
due to the probability bound of the chi-squared distribution (Lemma \ref{lem:chi2}(b)). It follows that
$$\P(\Omega_1(\eta)) \le \sum_{t=s+1}^{(K+1)s} p^{t-s} \times p^{-(1+\eta/2)(t-s)} \le \sum_{j=1}^{Ks} p^{-\eta j/2} \asymp p^{-\eta/2}.$$
On the other hand, due to Lemma \ref{lem:chi2}(a), $\P(\Omega_2) \le e^{-3n/8}$. Collecting these two pieces together completes the proof.
\end{proof}

\begin{proof}[Proof of Lemma \ref{lem:ratio}]
For simplicity of notation, define
\begin{align*}
\mc{A} &:= \left[\sigma_\star^2 \times \frac{1-M_1\epsilon_n}{1+M_1\epsilon_n}, \sigma_\star^2 \times \frac{1+M_1\epsilon_n}{1-M_1\epsilon_n}\right],\\ 
\mc{B} &:= \left[-\sigma z_{0n},+\sigma z_{0n}\right],\\
\mc{C} &:= \{\bm{v} \in \mathbb{R}^s: \Vert \bm{v} - \bm{\beta}^\star_{\xi_\star} \Vert \le M_2\sigma_\star\epsilon_n\},\\
\mc{C}' &:= \left\{\bm{v} \in \mathbb{R}^s: \Vert \bm{v} - \bm{\beta}^\star_{\xi_\star}\Vert \le \sqrt{M_2^2\sigma_\star^2\epsilon_n^2 -(t-s)\sigma^2 z_{0n}^2}\right\} \subseteq \mc{C}.
\end{align*}
Write
\begin{align}
\frac{\Pi(\widetilde{\Theta}|\mbf{X},\mbf{Y},\xi)}{\Pi(\widetilde{\Theta}|\mbf{X},\mbf{Y},\xi_\star)}
&= \frac{\int_{\mc{A}} \int_{\mc{B}^{p-t}}\Pi(\widetilde{\Theta} |\mbf{X},\mbf{Y},\sigma, \xi, \bm{\beta}_{\xi^c})h_0(\bm{\beta}_{\xi^c}/\sigma)d(\bm{\beta}_{\xi^c}/\sigma) g(\sigma^2)d\sigma^2 }{\int_{\mc{A}} \int_{\mc{B}^{p-t}} \Pi(\widetilde{\Theta} |\mbf{X},\mbf{Y},\sigma, \xi_\star, \bm{\beta}_{\xi^c}) h_0(\bm{\beta}_{\xi^c}/\sigma)d(\bm{\beta}_{\xi^c}/\sigma)g(\sigma^2)d\sigma^2} \nonumber\\
&\le \sup_{\sigma^2 \in \mc{A},\bm{\beta}_{\xi^c} \in \mc{B}^{p-t}} \frac{\Pi(\widetilde{\Theta}|\mbf{X},\mbf{Y},\sigma, \xi, \bm{\beta}_{\xi^c})}{\Pi(\widetilde{\Theta}|\mbf{X},\mbf{Y},\sigma, \xi_\star, \bm{\beta}_{\xi^c})}. \label{eqn:ratio}
\end{align}
For any $\sigma^2 \in \mc{A},~\bm{\beta}_{\xi^c} \in \mc{B}^{p-t}$, the numerator of \eqref{eqn:ratio} is bounded from above as
\begin{align*}
&= \int_{\bm{\beta}_\xi: \Vert \bm{\beta}_\xi - \bm{\beta}^\star_\xi \Vert \le M_2\sigma_\star\epsilon_n} \mc{N}(\mbf{Y}|\mbf{X}\bm{\beta}, \sigma^2\mbf{I})h_1(\bm{\beta}_\xi/\sigma)d(\bm{\beta}_\xi/\sigma)\\
&\le \left[\sup_{\bm{\beta}_\xi:~\Vert \bm{\beta}_\xi - \bm{\beta}^\star_\xi \Vert \le M_2\sigma_\star\epsilon_n} h_1(\bm{\beta}_\xi/\sigma) \right] \int \mc{N}(\mbf{Y}|\mbf{X}\bm{\beta}, \sigma^2\mbf{I})d(\bm{\beta}_\xi/\sigma)\\
&\le \sup_{\bm{z} \in \mc{C}/\sigma} h_1(\bm{z}) \times \left[\sup_z h_1(z)\right]^{t - s} \times \int \mc{N}(\mbf{Y}|\mbf{X}\bm{\beta}, \sigma^2\mbf{I})d(\bm{\beta}_\xi/\sigma). 
\end{align*}
On the other hand, the denominator of \eqref{eqn:ratio} is bounded from below as
\begin{align*}
&\ge \int_{\bm{\beta}_\xi: \bm{\beta}_{\xi \setminus \xi_\star} \in \mc{B}^{t-s},~\bm{\beta}_{\xi_\star} \in \mc{C}'} \mc{N}(\mbf{Y}|\mbf{X}\bm{\beta}, \sigma^2\mbf{I})h_1(\bm{\beta}_{\xi_\star}/\sigma)h_0(\bm{\beta}_{\xi\setminus\xi_\star}/\sigma)d(\bm{\beta}_\xi/\sigma)\\
&\ge \inf_{\bm{\beta}_{\xi \setminus \xi_\star} \in \mc{B}^{t - s}} \int_{\mc{C}'} \mc{N}(\mbf{Y}|\mbf{X}\bm{\beta}, \sigma^2\mbf{I})h_1(\bm{\beta}_{\xi_\star}/\sigma)d(\bm{\beta}_{\xi_\star}/\sigma) \\
&~~~\times \int_{\mc{B}^{t - s}} h_0(\bm{\beta}_{\xi \setminus \xi_\star}/\sigma)d(\bm{\beta}_{\xi \setminus \xi_\star}/\sigma)\\
&\ge \inf_{\bm{\beta}_{\xi \setminus \xi_\star} \in \mc{B}^{t - s}} \left[\inf_{\bm{\beta}_{\xi_\star} \in \mc{C}'} h_1(\bm{\beta}_{\xi_\star}/\sigma) \times \int \mc{N}(\mbf{Y}|\mbf{X}\bm{\beta}, \sigma^2\mbf{I})d(\bm{\beta}_{\xi_\star}/\sigma)\right]\\
&~~\times \int_{\mc{B}^{t - s}} h_0(\bm{\beta}_{\xi \setminus \xi_\star}/\sigma)d(\bm{\beta}_{\xi \setminus \xi_\star}/\sigma)\\
&\ge \inf_{\bm{z} \in \mc{C}/\sigma} h_1(\bm{z}) \times \left[\int_{\mc{B}/\sigma} h_0(z)dz\right]^{t-s} \times \inf_{\bm{\beta}_{\xi \setminus \xi_\star} \in \mc{B}^{t - s}} \int \mc{N}(\mbf{Y}|\mbf{X}\bm{\beta}, \sigma^2\mbf{I})d(\bm{\beta}_{\xi_\star}/\sigma).
\end{align*}
Comparison of the last two displays reveals that \eqref{eqn:ratio} is upper bounded by the product of three terms, which are bounded as follows.
\begin{enumerate}[label=(\roman*)]
\item Due to condition (a) of Theorem \ref{thm3},
\begin{align*}
\frac{\sup_{\bm{z} \in \mc{C}/\sigma} h_1(\bm{z})}{\inf_{\bm{z} \in \mc{C}/\sigma} h_1(\bm{z})}
&\le \frac{\sup_{\bm{z}: \Vert \bm{z} - \bm{\beta}^\star_{\xi_\star}/\sigma\Vert_1 \le M_2\sqrt{s}\sigma_\star\epsilon_n/\sigma} h_1(\bm{z})}{\inf_{\bm{z}: \Vert \bm{z} - \bm{\beta}^\star_{\xi_\star}/\sigma\Vert_1 \le M_2\sqrt{s}\sigma_\star\epsilon_n/\sigma} h_1(\bm{z})}\\
&\le e^{L_n \times 2M_2\sqrt{s}\sigma_\star\epsilon_n/\sigma} \lesssim e^{\eta\log p/2}.
\end{align*}
\item Due to Assumption \ref{asm1}(c),
$$\left(\frac{\sup_z h_1(z)}{\int_{\mc{B}/\sigma} h_0(z)dz}\right)^{t-s} \le \frac{\sup_z h_1(z)}{1-(t-s)e^{-n}} \le 2 \left(\sup_z h_1(z)\right)^{t-s},$$
for sufficiently large $n$ such that $n \ge \log(2Ks)$.
\item Some elementary calculus gives that
$$\frac{\int \mc{N}(\mbf{Y}|\mbf{X}\bm{\beta}, \sigma^2\mbf{I})d(\bm{\beta}_\xi/\sigma)}{\int \mc{N}(\mbf{Y}|\mbf{X}\bm{\beta}, \sigma^2\mbf{I})d(\bm{\beta}_{\xi_\star}/\sigma)} = \frac{\det{\left|2\pi (\mbf{X}_\xi^\T\mbf{X}_\xi)^{-1}\right|^{1/2}}}{\det{\left|2\pi (\mbf{X}_{\xi_\star}^\T\mbf{X}_{\xi_\star})^{-1}\right|^{1/2}}} \times \exp\left\{\frac{ R_{\xi_\star}^2 - R_{\xi}^2 }{2\sigma^2}\right\},$$
with $R_{\gamma}^2 = \Vert (\mbf{I}-\mbf{P}_\gamma)(\sigma_\star \bm{\varepsilon} - \mbf{X}_{\gamma^c}\bm{\beta}_{\gamma^c})\Vert^2$. By Lemma \ref{lem:schur} and Assumption \ref{asm2},
\begin{align*}
\frac{\det{\left|\mbf{X}_\xi^\T\mbf{X}_\xi\right|}}{\det{\left|\mbf{X}_{\xi_\star}^\T\mbf{X}_{\xi_\star}\right|}}
&= \det{\left|\mbf{X}_{\xi \setminus \xi_\star}^\T(\mbf{I}-\mbf{P}_{\xi_\star})\mbf{X}_{\xi \setminus \xi_\star} \right|}\\
&\ge \left[\lambda_{\min}\left(\mbf{X}_{\xi \setminus \xi_\star}^\T(\mbf{I}-\mbf{P}_{\xi_\star})\mbf{X}_{\xi \setminus \xi_\star}\right)\right]^{t-s}\\
&\ge \left[\lambda_{\min}\left(\mbf{X}_\xi^\T \mbf{X}_\xi\right)\right]^{t-s} \ge (n\lambda)^{t-s}.
\end{align*}
Conditionally on $\Omega(\eta)$ defined in Lemma \ref{lem:omega},
\begin{align*}
R_{\xi_\star}^2 - R_{\xi}^2
&= \Vert (\mbf{P}_\xi - \mbf{P}_{\xi_\star})(\sigma_\star \bm{\varepsilon} - \mbf{X}_{\xi^c}\bm{\beta}_{\xi^c})\Vert^2 + \Vert (\mbf{I} - \mbf{P}_{\xi_\star})\mbf{X}_{\xi \setminus \xi_\star}\bm{\beta}_{\xi \setminus \xi_\star}\Vert^2\\
&~~-2(\sigma_\star \bm{\varepsilon} - \mbf{X}_{\xi^c}\bm{\beta}_{\xi^c})^\T(\mbf{I} - \mbf{P}_{\xi_\star})\mbf{X}_{\xi \setminus \xi_\star}\bm{\beta}_{\xi \setminus \xi_\star}\\
&\le [\sigma_\star \Vert (\mbf{P}_\xi - \mbf{P}_{\xi_\star})\bm{\varepsilon}\Vert + \Vert \mbf{X}_{\xi^c}\bm{\beta}_{\xi^c}\Vert]^2 + \Vert \mbf{X}_{\xi \setminus \xi_\star}\bm{\beta}_{\xi \setminus \xi_\star}\Vert^2\\
&~~+ 2 \sigma_\star \Vert \bm{\varepsilon} \Vert \Vert \mbf{X}_{\xi \setminus \xi_\star}\bm{\beta}_{\xi \setminus \xi_\star} \Vert + 2 \Vert \mbf{X}_{\xi^c}\bm{\beta}_{\xi^c}\Vert \Vert \mbf{X}_{\xi \setminus \xi_\star}\bm{\beta}_{\xi \setminus \xi_\star}\Vert\\
&\le [\sigma_\star \sqrt{(2+\eta)(t-s)\log p} + \sigma \sqrt{n}pz_{0n}]^2 + [\sigma \sqrt{n} (t-s) z_{0n}]^2\\
&~~+ 4 \sigma_\star \sqrt{n} \times \sigma \sqrt{n}(t-s)z_{0n} + 2\sigma \sqrt{n}pz_{0n} \times \sigma\sqrt{n}(t-s)z_{0n}\\
&\le (2+3\eta/2)(t-s)\log p,
\end{align*}
for sufficiently large $n \ge N$, with $N$ not depending on $\xi$.
\end{enumerate}
Combining (i)-(iii) with \eqref{eqn:ratio} concludes the proof.
\end{proof}

\section{Preliminary Lemmas}

\begin{lemma}[{\citet[Theorem 1.1]{pelekis2016lower}}] \label{lem:pelekis}
For a Binomial distributed random variable $T \sim \mathrm{Binomial}(p, \mu)$, if $p \mu < t \le p -1$ then
$$\mbb{P}\left(T \ge t\right) \le \frac{\mu^{2(\tilde{t}+1)}}{2} \left. {p \choose \tilde{t}+1}  \right/ {t \choose \tilde{t}+1},$$
where $\tilde{t} = \lfloor (t - p\mu)/(1-\mu) \rfloor < t$.
\end{lemma}

\begin{lemma}[Probability Bounds of Chi-squared Distribution] \label{lem:chi2}
Let $\chi^2_d$ be a chi-squared random variable of degree $d$.
\begin{enumerate}[label=(\alph*)]
\item For any $\epsilon_n$ such that $n\epsilon_n > d_n$,
\begin{align*}
\mbb{P}(\chi^2_{n - d_n} / n  \ge 1 + \epsilon_n) &\le e^{- \min \left\{\frac{(n\epsilon_n+d_n)^2}{8(n - d_n)}, \frac{n\epsilon_n+d_n}{8}\right\}},\\
\mbb{P}(\chi^2_{n - d_n} / n  \le 1 - \epsilon_n) &\le e^{- \min \left\{\frac{(n\epsilon_n-d_n)^2}{8(n - d_n)}, \frac{n\epsilon_n-d_n}{8}\right\}}.
\end{align*}
In addition, if $\epsilon_n \to 0$ but $n\epsilon_n \succ d_n$,
\begin{align*}
\mbb{P}(\chi^2_{n - d_n} / n  \ge 1 + \epsilon_n) &\preceq e^{-(1/8 - o(1))n\epsilon_n^2},\\
\mbb{P}(\chi^2_{n - d_n} / n  \ge 1 + \epsilon_n) &\preceq e^{-(1/8 - o(1))n\epsilon_n^2}.
\end{align*}
\item
$$\mbb{P}(\chi^2_{d_n} \ge t_n) \le e^{-\left(\sqrt{2t_n - d_n} - \sqrt{d_n}\right)^2/4}.$$
In addition, if $t_n \succ d_n$ then for any $\tilde{t}_n$ such that $\tilde{t}_n/t_n \to 1$
$$\mbb{P}(\chi^2_{d_n} \ge t_n) \preceq e^{-[1/2 - o(1)]\tilde{t}_n}.$$
\end{enumerate}
\end{lemma}

\begin{proof}
For part (a), the first assertion follows from the sub-exponential tail of chi-squared distribution, and the second assertion is due to
\begin{align*}
[1/8 - o(1)]n\epsilon_n^2 &\preceq \frac{(n\epsilon_n+d_n)^2}{8(n - d_n)} \preceq \frac{n\epsilon_n+d_n}{8},\\
[1/8 - o(1)]n\epsilon_n^2 &\preceq \frac{(n\epsilon_n-d_n)^2}{8(n - d_n)} \preceq \frac{n\epsilon_n-d_n}{8}.
\end{align*}
For part (b), the first assertion is a corollary of \citet[Lemma 1]{laurent2000adaptive}, and the second assertion follows from
$$[1/2 - o(1)]\tilde{t}_n \preceq \left(\sqrt{2t_n - d_n} - \sqrt{d_n}\right)^2/4.$$
\end{proof}

\begin{lemma}\label{lem:split}
For a collection of subspace $\{\Theta_j\}_{j=1}^m$ and a collection of test functions $\{\varphi_j\}_{j=1}^m$
$$\sup_{\theta \in \cup_{j=1}^m \Theta_j} \mbb{E}_\theta \left(1 - \max_{j=1}^m \varphi_j\right) \le \max_{j=1}^m \left\{\sup_{\theta \in \Theta_j} \mbb{E}_\theta(1 - \varphi_j) \right\}.$$
\end{lemma}
\begin{proof}
\begin{align*}
\sup_{\theta \in \cup_{j=1}^m \Theta_j} \mbb{E}_\theta \left(1 - \max_{j=1}^m \varphi_j\right)
&= \max_{j=1}^m \left\{\sup_{\theta \in \Theta_j} \mbb{E}_\theta \left(1 - \max_{k=1}^m \varphi_k\right) \right\}\\
&= \max_{j=1}^m \left\{\sup_{\theta \in \Theta_j} \mbb{E}_\theta \left( \min_{k=1}^m (1-\varphi_k)\right) \right\}\\
&\le \max_{j=1}^m \left\{\sup_{\theta \in \Theta_j} \mbb{E}_\theta \left( 1-\varphi_j \right) \right\}.
\end{align*}
\end{proof}

\begin{lemma}[{\citet[Part of Corollary 2.4]{liu2005eigenvalue}}] \label{lem:schur}
Let $$\mbf{S} = \left[\begin{array}{c c} \mbf{S}_{11} & \mbf{S}_{12}\\ 
\mbf{S}_{21} & \mbf{S}_{22}\end{array}\right]$$ be a $p \times p$ positive semi-definite matrix with $q \times q$ non-singular principal sub-matrix $\mbf{S}_{11}$ then
$$\lambda_{\min}(\mbf{S}_{22} - \mbf{S}_{21}\mbf{S}_{11}^{-1}\mbf{S}_{12}) \ge \lambda_{\min}(\mbf{S}).$$
\end{lemma}

\end{document}